\documentclass[sn-mathphys-num]{sn-jnl}

\usepackage{indentfirst,enumerate,amssymb,amsfonts,amsmath,amsthm,mathrsfs,dsfont}
\usepackage{color}
\usepackage{graphicx}%
\usepackage{xcolor}%
\usepackage{comment}
\usepackage{verbatim}

\numberwithin{equation}{section}

\theoremstyle{thmstyleone}
\newtheorem{theorem}{Theorem}[section]
\newtheorem{proposition}[theorem]{Proposition}
\newtheorem{corollary}[theorem]{Corollary}

\theoremstyle{thmstyletwo}
\newtheorem{example}{Example}[section]

\theoremstyle{thmstylethree}
\newtheorem{definition}{Definition}[section]
\newcommand{\C}{\mathbb{C}}

\renewcommand{\Re}{\textrm{Re}}
\renewcommand{\Im}{\textrm{Im}}

\def\jp#1{{\left\langle{#1}\right\rangle}}

\begin{document}

	\title{Global hypoellipticity and global solvability of Vekua-type operators associated with diagonal operators on compact Lie groups}




\author*[1]{\fnm{R.} \sur{Paleari da Silva}}\email{ricardo.paleari@unespar.edu.br}

\affil*[1]{\orgdiv{Colegiado de Matem\'atica}, \orgname{Unespar}, \orgaddress{\street{Rua Comendador Correa Junior, 117}, \city{Paranaguá}, \postcode{83203-560}, \state{PR}, \country{Brazil}}}



\abstract{In this paper, we study Vekua-type operators associated with diagonal operators on compact Lie groups. Characterizations of global hypoellipticity and global solvability properties are presented on classes of Vekua-type operators with constant coefficients. We also present sufficient conditions in order to get global solvability for a class of Vekua-type operators with non-constant coefficients.	
}

\keywords{Global hypoellipticity, Vekua-type operators, Compact Lie groups, Fourier Analysis, 3-dimensional spheres}

\pacs[MSC Classification]{Primary 35H10; 35A01, Secondary 30G20, 43A75}

\maketitle

\section{Introduction}

The pioneer work \cite{GW1972_pams} present, for the first time, a relation between the global hypoellipticity (which we will refer as property (GH)) of a (real) constant coefficient vector field on the torus with a Diophantine condition on the coefficient of the operator, that is, a condition that relates the growth rate to which the coefficient can be approximated by rational numbers. A little later, the works \cite{GW1973_tams} and \cite{greenfield1973globally} suggest that the existence of real (GH) vector fields on a closed manifold is a topological constraint (see also \cite{Forni08_cont-math}). A very closed concept related to the (GH) property is the global solvability of the operator, which we will refer as (GS) condition. Since then, many mathematicians started to explore the existence of globally hypoelliptic vector fields, on even more general classes of operators, in different types of manifolds or different classes of regularity. For example, the works \cite{CH1977_pams}, \cite{BP1999_jmaa} and \cite{DM2016_mana} are concerned about the global hypoellipticity and global solvability of involutive systems of vector fields. Another commom direction is the study of perturbation of vector fields (or, more generally, perturbation of operators by lower order terms), trying to understand if (and how) these perturbations may affect the properties (GH) and (GS) (see for example \cite{AK2019_jst} and \cite{KMP2021_jde}). Inspired by \cite{vekua2014generalized}, one could consider perturbations which are $\mathbb{R}$-linear but not $\mathbb{C}$-linear, for example considering operators involving terms with the complex-conjugate of the variable. The works \cite{bergamasco2014solvability} and \cite{de2021solvability} explore this direction in the case of the torus $\mathbb{T}^n$ and defined which are now called ``Vekua-type operators''. Many of these works are set in the torus $\mathbb{T}^n$, where the Fourier-Series is a powerful tool for solving equations and characterizing the regularity of solutions. However, coming back to the original works of Greenfield-Wallach, the environment of compact Lie groups is an interesting place for studying these problems, since in them we also have a very well developed theory of Fourier-Series. Despite there is a general compact manifold approach, involving strongly-invariant operators that commutes with a fixed elliptic operator (like the Laplacian associated with a Riemannian metric in a general compact oriented manifold), the Fourier-Series on a compact Lie group via Peter-Weyl's Theorem is a theory more intrinsic to the Lie group. In this way, this approach is a really good way to produce explicity examples of operators which are (GH)/(GS) (or not). Based on \cite{RT2010_book}, some works were done in this direction (see for example \cite{KMR2020_bsm}, \cite{KMR2021_jfa} and \cite{KKM2024}). Then, the first work that combines Vekua-type operators and the approach of Peter-Weyl Fourier-Series on a compact Lie group is \cite{de2022regularity}, focused on constant-coefficient operators of order 1. Very recently, the work \cite{kirilov2026solvability} generalizes some of previous works, giving examples of more classes of (GS) Vekua-type operators on general compact Lie groups. However, this work is also focused on order $1$ operators. The idea of this work is to present natural extensions of \cite{de2022regularity} and \cite{kirilov2026solvability}, proving more classes of (GH) and (GS) operators on compact Lie groups that can be of any order and explicitly constructed. The main is based on \cite{da2025diagonal}, which increases in a natural way the amount of classes of examples of operators, dealing with, essentially, Fourier-multipliers on compact Lie groups and the Peter-Weyl Fourier-Series. This work is organized as follows. On section \ref{overview} we review the basic topics about Fourier Series on compact Lie groups and how this series characterize distribution and smooth functions. On section \ref{Section_Vekua_operators} we introduce Vekua-type operators of constant coefficient associated with diagonal operators on compact Lie groups and stabilish the main results that characterize global hypoellipticity and global solvability for this class. Finally, on section \ref{nonconstant} we introduce a class of non-constant coefficient Vekua-type operators associated with diagonal operators on compact Lie groups. We present the main result of this paper, the Theorem \ref{mainthm}, which gives sufficient conditions for finding globally solvable operators on this class.

\section{Overview on Fourier Analysis in Compact Lie groups}\label{overview}

In this section, we introduce the notation and fundamental results needed for this study. A more detailed presentation of these concepts, as well as the proofs of the results discussed here, can be found in \cite{RT2010_book}.

Let $G$ be a compact Lie group, and let $\mu$ denote the normalized Haar measure on $G$. The set of continuous irreducible unitary representations of $G$ will be denoted by $\textrm{Rep}(G)$. The quotient $\widehat{G}:=\textrm{Rep}(G)/\sim$, which identify isomorphic representations, is called the unitary dual of $G$. It is well known that $\widehat{G}$ is countable. By the Peter-Weyl Theorem, there exists an orthonormal basis for $L^2(G)$, which can be constructed as follows: for each class $\Xi \in \widehat{G}$, we select a representative matrix-valued function $\xi: G \to U(d_\xi)$, where $d_\xi$ is the dimension of the representation $\xi$. Writing $\xi = (\xi_{mn})_{1 \leq m,n \leq d_\xi}$, the set
$$
\bigcup_{\Xi \in \widehat{G}} \left\{ \sqrt{d_\xi} \cdot \xi_{mn}; \ 1 \leq m,n \leq d_\xi \right\}
$$
forms an orthonormal basis for $L^2(G)$. From now on, we assume that a unique representative $\xi$ has been chosen for each class $\Xi \in \widehat{G}$, although in special cases we may impose additional properties on these representatives.

Next, we can consider the Fourier analysis on $G$ with respect to this basis. For each $f \in L^1(G)$ and $\Xi = [\xi] \in \widehat{G}$, the $\xi$-Fourier coefficient of $f$ is given by the matrix
$$
\widehat{f}(\xi) \doteq \int_G f(g) \xi(g)^\ast \, d\mu(g).
$$
This is defined up to conjugation by unitary matrices, which is sufficient for our purposes, as we are primarily concerned with estimating the Hilbert-Schmidt norm of $\widehat{f}(\xi)$. 

We denote by $C^\infty(G)$ the space of smooth functions on $G$, equipped with the standard topology of uniform convergence for functions and their derivatives. The space of distributions on $G$ is denoted by $\mathcal{D}'(G)$, which is the topological dual of $C^\infty(G)$.

Let $\Delta$ be the Laplace-Beltrami operator on $G$. For each $\Xi = [\xi] \in \widehat{G}$, the matrix entries $\xi_{mn}$ are eigenfunctions of $\Delta$, all corresponding to the same eigenvalue $\lambda(\xi) \leq 0$. The operator $(I - \Delta)^{1/2}$ is positive definite, and we denote its eigenvalue corresponding to $\Xi = [\xi]$ by $\jp{\xi} \doteq (1 + \lambda(\xi))^{1/2}$.

For each $u \in \mathcal{D}'(G)$ and $\Xi = [\xi] \in \widehat{G}$, the $\xi$-Fourier coefficient of $u$ is defined by
$$
\widehat{u}(\xi) \doteq \langle u, \xi^\ast \rangle.
$$
Again, this is well-defined up to conjugation by a unitary matrix.

The Peter-Weyl basis allows us to characterize distributions and different regularity classes of functions through the behavior of their Fourier coefficients. For instance, suppose that for each $\Xi = [\xi] \in \widehat{G}$, we can associate a matrix $(x(\xi)_{mn}) \in \C^{d_\xi \times d_\xi}$, and that there exist constants $M > 0$ and $N > 0$ such that
\begin{equation}\label{dist}
	\big|x(\xi)_{mn} \big| \leq M \jp{\xi}^{N}
\end{equation}
for all $1 \leq m,n \leq d_\xi$. Then, the series
$$
u \doteq \sum_{\Xi = [\xi] \in \widehat{G}} d_\xi \sum_{m,n} x(\xi)_{mn} \xi_{nm}
$$
converges in $\mathcal{D}'(G)$ and defines a distribution $u$ such that $\widehat{u}(\xi)_{mn} = x(\xi)_{mn}$ for all $1 \leq m,n \leq d_\xi$, and $\Xi = [\xi] \in \widehat{G}$.

Conversely, if $u \in \mathcal{D}'(G)$, then there exist constants $M > 0$ and $N > 0$ such that
$$
\big|\widehat{u}(\xi)_{mn}\big| \leq M \jp{\xi}^{N}
$$
for all $1 \leq m,n \leq d_\xi$ and $\Xi = [\xi] \in \widehat{G}$.

Regarding smooth functions, a distribution $u \in \mathcal{D}'(G)$ is a smooth function if and only if, for every $N > 0$, there exists a constant $M > 0$ such that
\begin{equation}\label{smooth1}
	\big|\widehat{u}(\xi)_{mn} \big| \leq M \jp{\xi}^{-N}
\end{equation}
for all $1 \leq m,n \leq d_\xi$ and $\Xi = [\xi] \in \widehat{G}$. In this case, the corresponding Fourier series converges in the $L^2(G)$-norm, and the Plancherel formula holds.

For a continuous operator $P: \mathcal{D}'(G) \to \mathcal{D}'(G)$, the symbol of $P$ is defined at $g \in G$ and $\Xi = [\xi] \in \widehat{G}$ by the matrix
$$
\sigma_P(g,\xi) \doteq \xi(g)^\ast (P\xi)(g),
$$
where $P\xi$ is the matrix $(P\xi)_{mn} \doteq P(\xi_{mn})$ for $1 \leq m,n \leq d_\xi$. This is also defined up to conjugation by a unitary matrix. In the special case where $P$ is left-invariant (i.e., it commutes with left translations), $\sigma_P$ does not depend on $g \in G$, and for all $u \in \mathcal{D}'(G)$ and $\Xi = [\xi] \in \widehat{G}$, the following formula holds:

\begin{equation} \label{multiply}
\widehat{Pu}(\xi) = \sigma_P(\xi)\cdot \widehat{u}(\xi).
\end{equation} 

If $u \in \mathcal{D}'(G)$, we define $\overline{u}$ by the pairing:

$$\jp{\overline{u},\varphi}:= \overline{ \jp{u, \overline{\varphi}}}$$
for each $\varphi \in C^\infty(G)$. It is easy to see that $\overline{u} \in \mathcal{D}'(G)$. 

If $\xi \in \textrm{Rep}(G)$ , then $\overline{\xi}(g):=\overline{\xi(g)}$, $g \in G$, also defines an element of $\textrm{Rep}(G)$, called the conjugated representation of $\xi$. It is clear that $d_{\overline{\xi}}=d_{\xi}$.  Besides,

$$\widehat{\overline{u}}(\xi)_{k\ell}= \jp{\overline{u}, \overline{\xi}_{\ell k}} = \overline{  \jp{u, \xi_{\ell k}}}=\overline{\jp{u, ((\overline{\xi})^*)_{k \ell}}}=\overline{\widehat{u}(\overline{\xi})_{k\ell}}$$
for all $1 \leq k,\ell \leq d_\xi$, so

\begin{equation}\label{conj}
\widehat{\overline{u}}(\xi) = \overline{ \widehat{u}(\overline{\xi})}
\end{equation}
for all $[\xi] \in \widehat{G}$.

\section{Vekua-type operators}\label{Section_Vekua_operators}
Let $G$ be a compact Lie group and $L: \mathcal{D}'(G) \rightarrow \mathcal{D}'(G)$ be a left-invariant continuous operator such that $L(C^\infty(G)) \subset C^\infty(G)$. We also consider $p,q \in \mathbb{C}$, $p \neq 0$, and define the operator $P: \mathcal{D}'(G) \rightarrow \mathcal{D}'(G)$ by
\begin{equation}\label{vekua_op} 
Pu:= Lu -q\cdot u - p \cdot \overline{u},
\end{equation} 
for each $u \in \mathcal{D}'(G)$. We will call the operator $P$ as a Vekua-type operator associated with $L$. Note that $P$ is $\mathbb{R}$-linear but not $\mathbb{C}$-linear. 

We are concerned about regularization properties of $P$, in the sense of the following definition.

\begin{definition} We say that a continuous operator $P: \mathcal{D}'(G) \rightarrow \mathcal{D}'(G)$ which preserves smooth functions is globally hypoelliptic (or just (GH)) if the conditions $u \in \mathcal{D}'(G)$ and $Pu \in C^\infty(G)$ always imply $u \in C^\infty(G)$.
\end{definition}

For a certain class of operators $L$, this work will present a characterization for global hypoellipticity of the Vekua-type operator defined in (\ref{vekua_op}) in terms of certain ``Diophantine-condition'' and conditions about the coefficients $p,q$. 

Suppose that $u,f \in \mathcal{D}'(G)$ are such that $Pu=f$. By property (\ref{multiply}), for all $[\xi] \in \widehat{G}$ we have:
$$\widehat{Pu}(\xi) = \widehat{Lu}(\xi)-q \cdot \widehat{u}(\xi)-p \cdot \widehat{\overline{u}}(\xi)=(\sigma_L(\xi)-q. \textrm{Id})\cdot \widehat{u}(\xi)-p \cdot \widehat{\overline{u}}(\xi)=\widehat{f}(\xi).$$
Now, applying the same idea for the representation $\overline{\xi}$, by property (\ref{conj}) we have
$$\widehat{Pu}(\overline{\xi})=(\sigma_L(\overline{\xi})-q \cdot \textrm{Id}) \cdot \widehat{u}(\overline{\xi}) - p \cdot \widehat{ \overline{u}}(\overline{\xi}) =(\sigma_L(\overline{\xi})-q \cdot \textrm{Id}) \cdot \overline{ \widehat{ \overline{u}}(\xi)}-p \cdot \overline{ \widehat{u}(\xi)} = \widehat{f}(\overline{\xi})=\overline{ \widehat{\overline{f}}(\xi)}. $$
By taking the complex conjugate on the last equation above we get
$$-\overline{p} \cdot \widehat{u}(\xi) + (\overline{\sigma_L(\overline{\xi})}-\overline{q} \cdot \textrm{Id}) \cdot \widehat{ \overline{u}}(\xi) = \widehat{\overline{f}}(\xi).$$  
In this way, for all $[\xi] \in \widehat{G}$ we get a system on the matrix coefficients $\widehat{u}(\xi)$ and $\widehat{\overline{u}}(\xi)$ given by
\begin{equation}\label{mainsystem}
	\left \{ \begin{array}{c} (\sigma_L(\xi)-q. \textrm{Id})\cdot \widehat{u}(\xi)-p \cdot \widehat{\overline{u}}(\xi)=\widehat{f}(\xi) \\ -\overline{p} \cdot \widehat{u}(\xi) + (\overline{\sigma_L(\overline{\xi})}-\overline{q} \cdot \textrm{Id}) \cdot \widehat{ \overline{u}}(\xi) = \widehat{\overline{f}}(\xi)   \end{array}   \right.
\end{equation}

Now, we will restrict ourselves to a particular class of operators $L$ in a way we can actually treat the system (\ref{mainsystem}) above.

\begin{definition}\label{diagonal}
	We say that a continuous left-invariant operator $L: \mathcal{D}'(G) \rightarrow \mathcal{D}'(G)$ which preserves smooth functions is diagonal if the following conditions holds:
	\begin{itemize}
		\item[a)] for each $\Xi \in \widehat{G}$ there is a matrix representative $\xi \in \Xi$ such that $\sigma_L(\xi)$ is a diagonal matrix and that $\sigma_L(\overline{\xi})=\overline{\sigma_L(\xi)}$. In this case, we will denote the entries of this matrix $\sigma_L(\xi)$ by 
	$\sigma_L(\xi)=\textrm{diag}(\sigma_1(\xi),...,\sigma_{d_\xi}(\xi))$. 
		\item[b)] There are constants $M>0$ and $K \in \mathbb{N}$ such that
		\begin{equation}\label{condsimbolo}
		|\sigma_m(\xi)| \leq C \jp{\xi}^K
		\end{equation}
		for all $[\xi] \in \widehat{G}$ and $1 \leq m \leq d_\xi$.
	\end{itemize}
\end{definition} 

\begin{example} Let $k,n \in \mathbb{N}$ and $\{X_1,...,X_n\}$ be a finite family of (real) left-invariant vector fields on a compact Lie group $G$. For each multi-index $\alpha=(\alpha_1,...,\alpha_n) \in \mathbb{N}_0^n$ we define the operator $X^\alpha:= X_1^{\alpha_1} \circ ... \circ X_n^{\alpha_n}$. If for each $\alpha \in \mathbb{N}^n$, with $1 \leq |\alpha| \leq k$, a complex number $a_\alpha$ is given, then 
	$$L: = \sum_{|\alpha| \leq k} a_\alpha X^\alpha,$$
defines a continuous left-invariant operator on $\mathcal{D}'(G)$ which preserves smooth functions on $G$. Moreover, since $k \geq 1$, it is easy to see that $\sigma_{X_j}(\overline{\xi}) = \overline{\sigma_{X_j}(\xi)}$ for all $j=1,...,n$, which implies that $\sigma_{L}(\overline{\xi}) = \overline{\sigma_L(\xi)}$ for all $[\xi] \in \widehat{G}$.

	It is well known that $i \cdot X_j$ is a symmetric operator on $L^2(G)$ for each $j=1,...,n$. In particular, for each $\Xi \in \widehat{G}$ there is a matrix representative $\xi \in \Xi$ such that $\sigma_{X_j}(\xi)=\textrm{diag}(i\cdot \mu_1^j(\xi),...,i \cdot \mu_{d_\xi}^j(\xi))$, with $\mu_\ell^j(\xi) \in \mathbb{R}$ for all $\ell=1,...,d_\xi$. In this way, if $n=1$, then $L$ is a diagonal operator. If $n>1$ and all the vector fields $X_j$ commute with each other, then $L$ is a diagonal operator.   
\end{example}

Now suppose that $L$ is a diagonal operator on $G$, $p,q \in \mathbb{C}$, $p \neq 0$, and let $P$ be the corresponding Vekua-type operator as defined in (\ref{vekua_op}). Once for all, for each $\Xi \in \widehat{G}$ assume that we choose $\xi \in \Xi$ such that $\sigma_L(\xi)$ have the properties as in (\ref{diagonal}). In this way, for all $[\xi] \in \widehat{G}$ and $1 \leq k,\ell \leq d_\xi$, we have
 
\begin{equation}\label{subsystem}
	\left \{ \begin{array}{c} (\sigma_k(\xi)-q)\cdot \widehat{u}(\xi)_{k\ell}-p \cdot \widehat{\overline{u}}(\xi)_{k\ell}=\widehat{f}(\xi)_{k\ell} \\ -\overline{p} \cdot \widehat{u}(\xi)_{k\ell} + (\sigma_k(\xi)-\overline{q} ) \cdot \widehat{ \overline{u}}(\xi)_{k\ell} = \widehat{\overline{f}}(\xi)_{k\ell}   \end{array}   \right.
\end{equation}

The discrimant of this system will by denote by $\Delta_k(\xi)$, and is given by
\begin{equation}\label{discriminant} 
\Delta_k(\xi)=(\sigma_k(\xi)-q)\cdot (\sigma_k(\xi)-\overline{q})-|p|^2=\sigma_k(\xi)^2-2\sigma_k(\xi)\Re(q)+|q|^2-|p|^2
\end{equation} 

By Cramer's rule, it follows that
\begin{equation}\label{cramer}
\Delta_k(\xi) \cdot \widehat{u}(\xi)_{k\ell} =(\sigma_k(\xi)-\overline{q})\cdot \widehat{f}(\xi)_{k\ell} +p \cdot \widehat{\overline{f}}(\xi)_{k\ell}. 
\end{equation}

Here we assume that $L$ is an operator associated with a symbol $\sigma_L$ of order $m>0$. In this way, there is a constant $C_\sigma>0$ such that
$$|\sigma_k(\xi)| \leq C_\sigma \cdot \jp{\xi}^m$$
for all $[\xi] \in \widehat{G}$ and $1 \leq k \leq d_\xi$.

Next, we introduce a condition on $\Delta_k(\xi)$ that recalls Diophantine condition like in works .... (fazer referências).

\begin{definition}\label{DC} Let $L$ be a diagonal operator, $p,q \in \mathbb{C}$, $p \neq 0$, and $P$ the associated Vekua-type operator. We say that $P$ satisfy the condition (DC) if there exists $C>0$ and $M>0$ such that
	$$|\Delta(\xi)_k| \geq C \cdot \jp{\xi}^{-M}.$$
for all $[\xi] \in \widehat{G}$ with $\jp{\xi} \geq M$ and $1\leq k \leq d_\xi$. 
\end{definition}

\begin{proposition}\label{firstprop} If $P$ satisfy the condition (DC), then $P$ is (GH).
\end{proposition}

\begin{proof} Let $C>0$ and $M>0$ be the constants from the definition of the condition (DC). Suppose that $f \in C^\infty(G)$ and $u \in \mathcal{D}'(G)$ are such that $Pu=f$. We want to show that $u \in C^\infty(G)$. Fix $N>0$. By defining
	$$C_1(N):= \max \{|\widehat{u}(\xi)_{k\ell}|\cdot \jp{\xi}^N; \jp{\xi}\leq M, 1 \leq k,\ell \leq d_\xi\} \geq 0,$$
we have
$$|\widehat{u}(\xi)_{k\ell}| \leq C_1(N)\cdot \jp{\xi}^{-N}$$
for all $[\xi] \in \widehat{G}$ with $\jp{\xi} \leq M$ and $1 \leq k,\ell \leq d_\xi$.	
	
Since $f \in C^\infty(G)$, we also have $\overline{f} \in C^\infty(G)$, and then, there are constants $C_2(N)>0$ and $C_3(N)>0$ such that
$$|\widehat{f}(\xi)_{k\ell} | \leq C_2(N)\cdot \jp{\xi}^{-N-m-M}$$
and
$$|\widehat{ \overline{f}}(\xi)_{k\ell}| \leq C_3(N) \cdot \jp{\xi}^{-N-M}$$
for all $[\xi] \in \widehat{G}$, $1 \leq k,\ell \leq d_\xi$. So for any $[\xi] \in \widehat{G}$ with $\jp{\xi}\geq M$ and $1 \leq k,\ell \leq d_\xi$, we have 
\begin{align} 
	|\widehat{u}(\xi)_{k\ell}| & \leq  \frac{1}{|\Delta_k(\xi)|} \cdot \left[(|\sigma_k(\xi)|+|q|)|\widehat{f}(\xi)_{k\ell}|+ |p|| \widehat{\overline{f}}(\xi)_{k\ell}| \right]  \nonumber \\
	& \leq  \frac{1}{C}\jp{\xi}^M \left [ (C_\sigma \jp{\xi}^m+|q|)C_2(N) \jp{\xi}^{-N-M-m}+|p|C_3(N) \jp{\xi}^{-N-M}  \right ] \nonumber \\
	&  \leq \frac{C_\sigma\cdot C_2(N)}{C}\jp{\xi}^{-N}+ \frac{|q|C_2(N)}{C}\jp{\xi}^{-N-m}+\frac{|p|C_3(N)}{C}\jp{\xi}^{-N} \nonumber \\
	& \leq \frac{1}{C}\left[ C_\sigma \cdot C_2(N)+|q|\cdot C_2(N)+|p|\cdot C_3(N)\right]\cdot \jp{\xi}^{-N}. 
\end{align} 
Now, if $C(N):= \max  \{ 1, (C_\sigma \cdot C_2(N)+|q|\cdot C_2(N)+|p|\cdot C_3(N))/C\}$, then $C(N)>0$ depends only on $N$ and
$$|\widehat{u}(\xi)_{k\ell}| \leq C(N)\cdot \jp{\xi}^{-N}$$
for all $[\xi] \in \widehat{G}$ and $1 \leq k,\ell \leq d_\xi$, which guarantees that $u \in C^\infty(G)$.

\end{proof}

Note that if the condition (DC) holds, in particular $\Delta_k(\xi)$ vanish for at most a finite number of representations $\xi$ and $k$ such that $1 \leq k \leq d_\xi$. We will see that if the representations of the group are never self-dual, that is $[\xi] \neq [\overline{\xi}]$ for all $\xi  \in \textrm{Rep}(G)$, then this is a necessary condition for (GH) to hold. To see this, first let us introduce the following set:

\begin{equation}\label{setZ} 
	\mathcal{Z}=\{[\xi] \in \widehat{G}; \Delta_k(\xi)=0 \textrm{ for some } 1\leq k \leq d_\xi\}.
\end{equation}

\begin{proposition}\label{Zinfinite} Suppose that $[\xi] \neq [\overline{\xi}]$ for all $\xi \in \textrm{Rep}(G)$. If the set $\mathcal{Z}$ is infinite, then $P$ is not (GH).	
\end{proposition}  

\begin{proof} Suppose that $\mathcal{Z}$ is infinite. So, for each $n \in \mathbb{N}$ there exists $[\xi_n] \in \widehat{G}$ and $1 \leq k_n \leq d_{\xi_n}$ such that $\Delta_{k_n}(\xi_n)=0$. We can also suppose that $[\xi_n] \neq [{\xi_{\tilde{n}}}] $ and 
$[\xi_n] \neq [\overline{\xi_{\tilde{n}}}] $ for all $n \neq \tilde{n}$. Define the sequence of matrices $(x(\xi))_{[\xi] \in \widehat{G}}$ given by
$$
x(\xi)_{k\ell} = \left \{ \begin{array}{ccc} \sigma_{k_n}(\xi_n)-\overline{q} & \textrm{if} & [\xi]=[\xi_n] \textrm{ and } k=\ell=k_n \\ p & \textrm{if} & [\xi]=[\overline{\xi_n}] \textrm{ and } k=\ell=k_n \\ 0 &  & \textrm{otherwise}. \end{array}   \right.
$$
This sequence defines a distribution $u \in \mathcal{D}'(G)$. In fact, for all $n$ we have
$$|x(\xi_n)_{k_n k_n}| = |\sigma_{n_k}(\xi_n)-\overline{q}|\leq C_\sigma \jp{\xi_n}^m+|q| \leq (C_\sigma+|q|)\jp{\xi_n}^m $$
and
$$|x(\overline{\xi_n})_{k_n k_n}| \leq |p| < |p| \jp{\overline{\xi_n}}^m,$$
so, if we define $C:=\max\{C_\sigma+|q|,|p|\}>0$, then 
$$|x(\xi)_{k\ell}| \leq C\cdot \jp{\xi}^m$$
for all $[\xi]\in \widehat{G}$ and $1 \leq k,\ell \leq d_\xi$. More than that, $u$ is not a smooth function because $|\widehat{u}(\overline{\xi_n})_{k_n k_n}| = |p| \neq 0$ for all $n \in \mathbb{N}$, so the Fourier coefficients of the distribution $u$ do not decay. We claim that $Pu=0$. Since $\widehat{u}(\xi) =0$ if $[\xi] \notin \{[\xi_n],[\overline{\xi_n}], n \in \mathbb{N}\}$, it is enough to verify that $\widehat{Pu}([\xi_n])=0$ and $\widehat{Pu}([\overline{\xi_n}])=0$ for all $n \in \mathbb{N}$. By (\ref{subsystem}) we have
\begin{eqnarray*} \widehat{Pu}(\xi_n)_{k_n k_n} & = & (\sigma_{k_n}(\xi_n)-q) \cdot \widehat{u}(\xi_n)_{k_n k_n}- p \cdot \overline{\widehat{u}(\overline{\xi_n})_{k_n k_n}} \\
	& = & (\sigma_{k_n}(\xi_n)-q)\cdot (\sigma_{k_n}(\xi_n)-\overline{q})- p \cdot \overline{p} \\
	& = & \Delta_{k_n}(\xi_n) \\
	& = & 0, 
\end{eqnarray*} 
and
\begin{eqnarray*}
	\widehat{Pu}(\overline{\xi_n})_{k_n k_n} & = & -\overline{p}\cdot \widehat{u}(\overline{\xi_n})_{k_n k_n} + (\sigma_{k_n}(\overline{\xi_n})-\overline{q})\cdot \widehat{ \overline{u}}(\overline{\xi_n})_{k_n k_n} \\ 
	& = & -\overline{p} \cdot p +(\overline{\sigma_{k_n}(\xi_n)}-\overline{q}) \cdot \overline{ \widehat{u}(\xi_n)} \\
	& = & -|p|^2+ (\overline{\sigma_{k_n}(\xi_n)-q})\cdot (\overline{ \sigma_{k_n}(\xi_n)-\overline{q}}) \\
	& = & \overline{ \Delta_{k_n}(\xi_n)} \\
	& = & 0
\end{eqnarray*}
for all $n\in \mathbb{N}$. So $u \in \mathcal{D}'(G) \setminus C^\infty(G)$ satisfies $Pu=0 \in C^\infty(G)$, which concludes that $P$ is not (GH).
\end{proof}

Next, we want to verify that under the hypothesis that every element $\xi \in \textrm{Rep}(G)$ is not self-dual, the condition (DC) is necessary for (GH).

\begin{proposition} Suppose that $[\xi] \neq [\overline{\xi}]$ for all $\xi \in \textrm{Rep}(G)$. If the Vekua-type operator $P$ is (GH), then the condition (DC) holds.
\end{proposition} 

\begin{proof} Suppose that condition (DC) does not hold. Then, for each $n \in \mathbb{N}$ there exists $[\xi_n] \in \widehat{G}$ and $1 \leq k_n \leq d_{\xi_n}$ such that $\jp{\xi_n}>n$ and $|\Delta_{k_n}(\xi_n)|<\jp{\xi_n}^{-n}$ for all $n \in \mathbb{N}$. By Proposition \ref{Zinfinite} we may also assume that $|\Delta_{k_n}(\xi_n)|>0$ for all $n \in \mathbb{N}$. Now consider the sequence of matrices $(x(\xi))_{[\xi] \in \widehat{G}}$ given by
	$$x(\xi)_{k\ell}:= \left\{ \begin{array}{ccc} \Delta(\xi)_{k} & \textrm{if} & [\xi]=[\xi_n] \textrm{ or } [\xi]=[\overline{\xi}_n] \textrm{ and } k=\ell=k_n\\ 0 & & \textrm{otherwise}   \end{array}  \right.$$
By hypothesis it is clear that $(x(\xi))_{[\xi]}$ is the sequence of Fourier-coefficients of a smooth function $f \in C^\infty(G)$. Now, for each $c \in \mathbb{C}\setminus \{0\}$ consider the smooth function $f_c:=c \cdot f$ and the sequence of matrices $(y(\xi))_{[\xi]}$ given by
$$y(\xi)_{k\ell}:= \left\{ \begin{array}{ccc} \sigma_{k}(\xi)\cdot c-\overline{q}\cdot c+p\cdot \overline{c} & \textrm{if} & [\xi]=[\xi_n] \textrm{ or } [\xi]=[\overline{\xi}_n] \textrm{ and } k=\ell=k_n\\ 0 & & \textrm{otherwise}   \end{array}  \right.$$
Notice that for $[\xi]=[\xi_n]$ or $[\xi]=[\overline{\xi_n}]$ and $k=\ell=k_n$ we have
$$|y(\xi)_{k \ell}| \leq |c| \cdot(|\sigma_k(\xi)|+|q|+|p|) \leq |c| \cdot (C_\sigma +|p|+|q|)\cdot \jp{\xi}^m,$$
which implies that $(y(\xi))_{[\xi]}$ is the sequence of Fourier-coefficients of a distribution $u_c \in \mathcal{D}'(G)$. We claim that $Pu_c=f_c$. To see this, it is sufficient to verify that $\widehat{Pu_c}(\xi_n)_{k_n k_n}=\widehat{f_c}(\xi_n)_{k_n k_n}$ and $\widehat{Pu_c}(\overline{\xi_n})_{k_n k_n}=\widehat{f_c}(\overline{\xi_n})_{k_n k_n}$  for all $n$, $k,\ell$. But, again by equations (\ref{subsystem}) we have
\begin{eqnarray*}
	\widehat{Pu_c}(\xi_n)_{k_n k_n} & = & (\sigma_{k_n}(\xi_n)-q) \cdot \widehat{u_c}(\xi_n)_{k_n k_n}-p\cdot \widehat{\overline{u_c}}(\xi_n)_{k_n k_n} \\
	& = & (\sigma_{k_n}(\xi_n)-q) \cdot (\sigma_{k_n}(\xi_n)\cdot c-\overline{q}\cdot c +p\cdot \overline{c})- p \cdot \overline{(\sigma_{k_n}(\overline{\xi_n})\cdot c-\overline{q}\cdot c+p\cdot \overline{c})} \\
	& = & (\sigma_{k_n}(\xi_n)-q) \cdot (\sigma_{k_n}(\xi_n)\cdot c-\overline{q}\cdot c +p\cdot \overline{c})-p \cdot (\sigma_{k_n}(\xi_n)\cdot \overline{c}-q \cdot \overline{c}+\overline{p}\cdot c) \\
	& = & c\cdot (\sigma_{k_n}(\xi_n)-q)\cdot (\sigma_{k_n}(\xi_n)-\overline{q})+p \cdot \overline{c}\cdot (\sigma_{k_n}(\xi_n)-q)- p \cdot \overline{c} \cdot(\sigma_{k_n}(\xi_n)-q) \\ 
	&   & -c \cdot |p|^2 \\
	& = & c \cdot \Delta_{k_n}(\xi_n) \\
	& = & \widehat{f_c}(\xi_n)_{k_n k_n}
\end{eqnarray*} 
for all $n \in \mathbb{N}$. Similarly, it can be verified that $\widehat{Pu_c}(\overline{\xi_n})_{k_n k_n}=\widehat{f_c}(\overline{\xi_n})_{k_n k_n}$, concluding that $Pu_c=f_c$. 
Notice that if $u_c \in C^\infty(G)$, in particular we must have $\lim_n |\widehat{u_c}(\xi_n)_{k_n}| = \lim_n |\widehat{u_c}(\overline{\xi_n})_{k_n}|=0$. But since
$$|\widehat{u}(\xi_n)_{k_n}| = |c| \cdot \left|\sigma_{k_n}(\xi_n)-\left(\overline{q}-\frac{p \cdot \overline{c}}{c}\right) \right| $$
and
$$|\widehat{u}(\overline{\xi_n})_{k_n}| = |c| \cdot \left|\overline{\sigma_{k_n}(\xi_n)}-\left(\overline{q}-\frac{p \cdot \overline{c}}{c}\right) \right| $$
we must have that $\overline{q}-p \cdot \overline{c}/c \in \mathbb{R}$. Using this observation, we are going to split the proof in some cases.

\begin{itemize} 
	
\item[1)] $\Im(\overline{q}-p) \neq 0$:

In this case we have that $u_1 \in \mathcal{D}'(G) \setminus C^\infty(G)$. In fact, if we suppose that $u_1 \in C^\infty(G)$, then by the observation we made above we must have $\overline{q}-p\cdot 1/1=\overline{q}-p \in \mathbb{R}$, which is a contradiction with our hypothesis. So $u_1 \in \mathcal{D}'(G) \setminus C^\infty(G)$, which implies that $P$ is not (GH).

\item[2)] $\Im(\overline{q}-p) = 0$:

This hypothesis means that $\Im(p)+\Im(q)=0$. This case will be subdivided into other two:

\item[2.1)] $\Im(p) \neq 0$;

For this case we will prove that $u_i \in \mathcal{D}'(G)\setminus C^\infty(G)$. In fact, if $u_i \in C^\infty(G)$, then by the observation above we must have $\overline{q}-p \cdot \overline{i}/i=\overline{q}+p \in \mathbb{R}$, which means that $\Im(q)=\Im(p)$. But since we are assuming that $\Im(p)+\Im(q)=0$, this implies that $\Im(p) =0$, which is a contradiction with our hypothesis. So $u_i \in \mathcal{D}'(G)\setminus C^\infty(G)$ and again $P$ is not (GH).

\item[2.2)] $\Im(p) =0$.

Finally, for this last case observe that since we are assuming $\Im(p)+\Im(q)=0$, the equation $\Im(p)=0$ implies that both $p,q$ are real numbers. We will see that $u_{1+ip} \in \mathcal{D}'(G)\setminus C^\infty(G)$. In fact, if this is not the case, then $\overline{q}-p \cdot (1-ip)/(1+ip)$ should be a real number. But

$$\overline{q}-p \cdot (1-ip)/(1+ip) = q-p \cdot \frac{1-2 p \cdot i -p^2}{1+p^2} =q-\frac{p(1-p^2)}{1+p^2}+\frac{2p^2}{1+p^2} \cdot i,$$
and since $p,q \in \mathbb{R}$, for this number to be real we must have $p=0$, which is a contradiction with our initial hypothesis about $p$. So $u_{1+ip} \in \mathcal{D}'(G)\setminus C^\infty(G)$ and again $P$ is not (GH). This concludes that if the condition (DC) does not hold, then $P$ is not (GH).
\end{itemize}
\end{proof}

\begin{corollary} Suppose that $[\xi] \neq [\overline{\xi}]$ for all $\xi \in \textrm{Rep}(G)$. The Vekua-type operator $P$ is (GH) if, and only if, the condition (\ref{DC}) holds.
\end{corollary}

Now we are going to study a notion of global solvability for the Vekua-type operator $P$. First, let $u,f \in \mathcal{D}'(G)$ such that $Pu=f$. By (\ref{cramer}) we have the following:
\begin{equation}\label{comp}
	\Delta_k(\xi)=0 \Rightarrow (\sigma_k(\xi)-\overline{q})\cdot \widehat{f}(\xi)_{k\ell}+p \cdot \overline{ \widehat{f}(\overline{\xi})_{k\ell}} =0.
\end{equation} 
So for a distribution $f$ to be in the image of the operator $P$, the condition above must hold. Because of that, we will consider the space
\begin{equation}\label{setA} \mathcal{A}:=\{ f \in \mathcal{D}'(G);(\ref{comp}) \textrm{ holds} \}
\end{equation} 
and call $\mathcal{A}$ the space of admissible distributions. 

\begin{definition} We say that the Vekua-type operator $P$ is globally solvable (or just (GS)) if $P(\mathcal{D}'(G)) = \mathcal{A}$.
\end{definition}

We will see the characterization for when $P$ is (GS) in terms of a similar condition as the condition (DC) in the case of non-self dual representations. So let us assume that $[\xi] \neq [\overline{\xi}]$ for all $\xi \in \textrm{Rep}(G)$.

Let $f \in \mathcal{A}$, $[\xi] \in \widehat{G}$ and $1 \leq k \leq d_\xi$. If $\Delta_k(\xi)\neq 0$, then, by (\ref{cramer}), any solution $u \in \mathcal{D}'(G)$ for $Pu=f$ must satisfy
$$\widehat{u}(\xi)_{k\ell} = \frac{1}{\Delta_k(\xi)}\cdot \left((\sigma_k(\xi)-\overline{q})\cdot \widehat{f}(\xi)_{k\ell}+p \cdot \overline{ \widehat{f}(\overline{\xi})_{k\ell}} \right).$$
Define a sequence of matrices $(x(\xi))_{[\xi] \in \widehat{G}}$ by
$$x(\xi)_{k\ell} :=\frac{1}{\Delta_k(\xi)}\cdot \left((\sigma_k(\xi)-\overline{q})\cdot \widehat{f}(\xi)_{k\ell}+p \cdot \overline{ \widehat{f}(\overline{\xi})_{k\ell}} \right)$$
if $\Delta_k(\xi) \neq 0$, and
$$\left \{ \begin{array}{ccc} x(\xi)_{k\ell} & = & 0 \\ x(\overline{\xi})_{k\ell} & = & -\frac{1}{\overline{p}} \cdot \overline{\widehat{f}(\xi)_{k\ell}} \end{array}\right.$$ 
if $\Delta_k(\xi)=0$. 

We will define a sufficient condition for this sequence of matrices to define a distribution on $G$.

\begin{definition} We say that the operator $P$ satisfy the condition (DC') if there exists $C>0$ and $M>0$ such that
	\begin{equation}\label{DC'} |\Delta_k(\xi)| \geq C\cdot \jp{\xi}^{-M}		
	\end{equation}
for all $[\xi] \in \widehat{G}$ and $1 \leq k \leq d_\xi$ such that $\Delta_k(\xi) \neq 0$. 	
\end{definition} 

Assume that (\ref{DC'}) holds and let $C$ and $M$ be the positive constants of the definition, we will see that the sequence $(x(\xi)_{[\xi]})$ defined before is the sequence of Fourier-coefficients of a distribution $u \in \mathcal{D}'(G)$. Since $f \in \mathcal{D}'(G)$, there exists $n$ and $C_f$ such that
$$|\widehat{f}(\xi)_{k\ell}| \leq C_f \cdot \jp{\xi}^n$$
for all $[\xi]\in \widehat{G}$ and $1 \leq k,\ell \leq d_\xi$. Now fix $[\xi] \in \widehat{G}$ and $1 \leq k,\ell \leq d_\xi$. If $\Delta_k(\xi)=0$ we have $x(\xi)_{k\ell}=0$ and no estimate is necessary in this case. On the other hand,
$$|x(\overline{\xi})_{k\ell}| = \frac{1}{|p|} |\widehat{f}(\xi)_{k\ell}| \leq \frac{C_f}{|p|}\cdot \jp{\overline{\xi}}^n<\frac{C_f}{|p|}\cdot \jp{\overline{\xi}}^{M+m+n}.$$
Now, if $\Delta_k(\xi) \neq 0$, then
\begin{eqnarray*} |x(\xi)_{k\ell}|& \leq & \frac{1}{|\Delta_k(\xi)|}\cdot \left[(|\sigma_k(\xi)|+|q|)|\widehat{f}(\xi)_{k\ell}|+|p|\cdot |\widehat{f}(\overline{\xi})_{k\ell}| \right] \\
	& \leq & \frac{1}{C} \jp{\xi}^M \cdot \left[ (C_\sigma+|q|)\jp{\xi}^m C_f \jp{\xi}^n+|p| C_f \jp{\xi}^n \right] \\
	& < & \frac{1}{C} \left( C_f\cdot (C_\sigma+|q|)+|p|\right) \cdot \jp{\xi}^{M+m+n}.
\end{eqnarray*} 
In this way, if $\widetilde{C}:= \max \{C_f/|p|, \left( C_f\cdot (C_\sigma+|q|)+|p|\right)/C\}>0$, then
$$|x(\xi)_{k\ell}| \leq \widetilde{C} \cdot \jp{\xi}^{M+m+n}$$
for all $[\xi]\in \widehat{G}$ and $1 \leq k,\ell \leq d_\xi$, so $(x(\xi))_{[\xi]}$ defines a distribuition $u \in \mathcal{D}'(G)$. 

By definition of $u$, it is clear that $\widehat{Pu}(\xi)_{k\ell}=\widehat{f}(\xi)_{k\ell}$ if $\Delta_k(\xi)\neq 0$. If $\Delta_k(\xi)=0$, then
\begin{eqnarray*} \widehat{Pu}(\xi)_{k\ell} & = & (\sigma_k(\xi)-q) \cdot \widehat{u}(\xi)_{k\ell}- p \cdot \overline{\widehat{u}(\overline{\xi})_{k\ell}} \\
	& = & (\sigma_k(\xi)-q) \cdot 0 -p \cdot \overline{ -\frac{1}{\overline{p}}\cdot \overline{f}(\xi)_{k\ell} } \\
	& = & \widehat{f}(\xi)_{k\ell}
\end{eqnarray*} 	
and
\begin{eqnarray*} \widehat{Pu}(\overline{\xi})_{k\ell} & = & (\overline{\sigma_k(\xi)}-q)\cdot \widehat{u}(\overline{\xi})_{k\ell} - p \cdot \overline{ \widehat{u}(\xi)_{k\ell}} \\
	& = & (\overline{\sigma_k(\xi)}-q)\cdot -\frac{1}{\overline{p}} \overline{ \widehat{f}(\xi)_{k\ell}} \\
	& = & -\frac{1}{\overline{p}} \overline{ (\sigma_k(\xi)-\overline{q})\cdot \widehat{f}(\xi)_{k\ell} } \\
	& = &  \frac{1}{\overline{p}}\cdot \overline{ p \cdot \overline{\widehat{f}(\overline{\xi})_{k\ell}}} \\
	& = & \widehat{f}(\overline{\xi})_{k\ell}, 
\end{eqnarray*} 
which guarantees that $Pu=f$ and $P$ is (GS). In the next proposition we are going to prove the converse.

\begin{proposition}\label{GSconstant} Assume that $[\xi] \neq [\overline{\xi}]$ for all $\xi \in \textrm{Rep}(G)$. The Vekua-type operator $P$ is (GS) if, and only if, the condition (\ref{DC'}) holds.
\end{proposition}

\begin{proof} We already proved that if condition (DC') holds, then $P$ is (GS). Assume now that (DC') does not hold. In this way, for each $n \in \mathbb{N}$ there exists $[\xi_n] \in \widehat{G}$ and $1 \leq k_n \leq d_{\xi_n}$ such that 
	$$0<|\Delta_{k_n}(\xi_n)|<\jp{\xi_n}^{-n}.$$	
Again we may assume that $[\xi_n] \neq [\xi_{\tilde{n}}]$ and $[\xi_n] \neq [\overline{\xi_{\tilde{n}}}]$ for all ${\tilde{n}} \neq n$. Consider the sequence of matrices $(x(\xi)_{[\xi]})$ given by
$$x(\xi)_{k\ell}:= \left \{ \begin{array}{ccc} \overline{p}^{-1} & \textrm{if} & \xi=\overline{\xi_n} \textrm{ and } k=\ell=k_n \\ 0 &  & \textrm{otherwise} \end{array}  \right.  $$ 
It is clear from the definition that $(x(\xi))_{[\xi]}$ defines a distribution $f$ and that $f \in \mathcal{A}$. We claim that there is no $u \in \mathcal{D}'(G)$ such that $Pu=f$. In fact, suppose by contradiction that such $u$ exists. Then, for all $n\in \mathbb{N}$ we must have
$$\Delta_{k_n}(\xi_n) \cdot \widehat{u}(\xi_n)_{k_n k_n} =(\sigma_{k_n}(\xi_n)-\overline{q})\cdot \widehat{f}(\xi_n)_{k_n k_n}+p \cdot \overline{ \widehat{f}(\overline{\xi_n})_{k_n k_n}}=0+p \cdot \frac{1}{p}=1,$$
which implies that
$$|\widehat{u}(\xi_n)_{k_n k_n}| = \frac{1}{|\Delta_{k_n}(\xi_n)|} > \jp{\xi_n}^n$$
for all $n \in \mathbb{N}$. So, the sequence $(\widehat{u}(\xi_n))_n$ does not have a moderate growth and cannot define a smooth distribution. In this way we conclude that $P$ is not (GS) and finish the proof. 
\end{proof}

A very interesting non-commutative Lie group, whose unitary dual can be explicitly described, is the $3$-sphere $\mathbb{S}^3 \doteq \{x \in \mathbb{R}^4; \|x\|_2 = 1\} \equiv SU(2)$. A classic result shows that the unitary dual $\widehat{\mathbb{S}^3}$ is in bijection with the set $\frac{1}{2}\mathbb{N}_0$, where for each $\ell \in \frac{1}{2}\mathbb{N}_0$, there is a unique, up to isomorphism, continuous irreducible unitary representation $t^\ell : \mathbb{S}^3 \to U(2\ell + 1)$. It is common to represent the entries of the matrix-valued function $t^\ell$ by $t^\ell_{mn}$, where $m, n \in J_\ell \doteq \{-\ell, -\ell + 1, -\ell + 2, \dots, \ell - 1, \ell\}$.

The Lie algebra $\mathfrak{su}(2)$ has a standard basis $\{Y_1, Y_2, Y_3\}$, which satisfies the commutation relations $[Y_1, Y_2] = Y_3$, $[Y_2, Y_3] = Y_1$, and $[Y_3, Y_1] = Y_2$. The left-invariant operators associated with these Lie algebra elements will be denoted by $D_1$, $D_2$, and $D_3$, respectively. It can be shown that $D_3$ satisfies
$$
D_3(t^\ell_{mn}) = -in \cdot t^\ell_{mn}
$$
for all $m, n \in J_\ell$ and $\ell \in \frac{1}{2}\mathbb{N}_0$. 

The action of the operators $D_1$ and $D_2$ on the functions $t^\ell_{mn}$ can also be expressed, but the formulas are more complicated. However, there is an alternative basis for $\mathfrak{su}(2)$ (over $\mathbb{C}$), denoted by $\{\partial_0, \partial_+, \partial_-\}$, where their action on the functions $t^\ell_{mn}$ becomes much simpler. Defining
$$
\partial_+ \doteq iD_1 - D_2, \quad \partial_- \doteq iD_1 + D_2, \quad \partial_0 \doteq iD_3,
$$
it is well known that:
\begin{align}\label{tabela}
	\partial_+(t^\ell_{mn})  = & -\sqrt{(\ell - n)(\ell + n + 1)} \cdot t^{\ell}_{m, n+1}, \nonumber \\
	\partial_-(t^\ell_{mn})  = & -\sqrt{(\ell + n)(\ell - n + 1)} \cdot t^{\ell}_{m, n-1}, \nonumber \\
	\partial_0(t^\ell_{mn})  = & n \cdot t^\ell_{mn},
\end{align}
for all $\ell \in \frac{1}{2}\mathbb{N}_0$ and $m, n \in J_\ell$.

One special property about $\mathbb{S}^3$ is that since there is exactly one representation of each dimension on $\widehat{\mathbb{S}^3}$, it follows that for all $\ell \in \frac{1}{2}\mathbb{N}_0$, the conjugated representation $\overline{\ell}$ is isomorphic to $\ell$. So we cannot apply all the previous results for Vekua-type operators in this case since we were assuming most of the time that $[\xi] \neq [\overline{\xi}]$ for all $[\xi]\in \widehat{G}$. However, in the case of the group $\mathbb{S}^3$ it can be proved that
$$\overline{t^\ell(x)_{mn}}=(-1)^{m-n} t^\ell(x)_{-m -n},$$ 
which implies that
\begin{equation}\label{conj}
	\widehat{u}(\overline{\ell})_{mn} = (-1)^{m-n} \widehat{u}(\ell)_{-m-n}
\end{equation}
for all $u \in \mathcal{D}'(\mathbb{S}^3)$, $\ell \in \frac{1}{2}\mathbb{N}_0$ and $m,n \in J_\ell$. 
This property will allow us to adapt previous proofs for Vekua-type operators on $\mathbb{S}^3$. Essentially, instead of defining singular solutions on representations $[\xi]$ and $[\overline{\xi}]$ separately, we will define them in a similar way on the entries $mn$ and $-m-n$.

Now let $L: \mathcal{D}'(\mathbb{S}^3) \rightarrow \mathcal{D}'(\mathbb{S}^3)$ be a continuous left-invariant operator that preserves smooth functions such that $\sigma_L(\ell)$ is a diagonal matrix for all $\ell \in \frac{1}{2}\mathbb{N}_0$. Let $p,q \in \mathbb{C}$ with $p \neq 0$ and $P: \mathcal{D}'(\mathbb{S}^3) \rightarrow \mathcal{D}'(\mathbb{S}^3)$ be the corresponding Vekua-type operator.

If $u \in \mathcal{D}'(\mathbb{S}^3)$, $\ell \in \frac{1}{2}\mathbb{N}_0$ and $m,n \in J_\ell$, then by property (\ref{conj}) we have
\begin{equation}\label{eqn1} (\sigma_m(\ell)-q) \cdot \widehat{u}(\ell)_{mn}-p \cdot (-1)^{m-n} \overline{\widehat{u}(\ell)_{-m-n}}=\widehat{Pu}(\ell)_{mn}.
\end{equation} 
By applying the equation above on the representation $\overline{\ell}$ and taking the complex-conjugate on it, again using property (\ref{conj}) we will get
\begin{equation}\label{eqn2} -\overline{p} \cdot \widehat{u}(\ell)_{mn}+(-1)^{m-n}(\overline{\sigma_{-m}(\ell)}-\overline{q}) \overline{\widehat{u}(\ell)_{-m -n}} = (-1)^{m-n} \overline{ \widehat{Pu}(\ell)_{-m -n}}.
\end{equation} 

The module of the discriminant of the system (\ref{eqn1}) and (\ref{eqn2}) is given by
$$|\Delta_m(\ell)| = | (\sigma_m(\ell)-q) \cdot (\overline{\sigma_{-m}(\ell)}-\overline{q})-|p|^2|.$$

The set $\mathcal{Z}$ defined in (\ref{setZ}) takes the form
$$\mathcal{Z} = \left\{\ell \in \frac{1}{2}\mathbb{N}_0; \exists m \in J_\ell \textrm{ such that } |\Delta_m(\ell)|=0\right\}.$$

We are going to start by proving an analogue of the Proposition \ref{Zinfinite} for this case.

\begin{proposition}\label{ZinfiniteSU2} If the set $\mathcal{Z}$ is infinite, then $P$ is not (GH).
\end{proposition}

\begin{proof} If $\mathcal{Z}$ is infinite, then there exists an increasing sequence $(\ell_j)_{j \in \mathbb{N}}$ in $\frac{1}{2}\mathbb{N}_0$ and $m_j \in J_{\ell_j}$ such that $\Delta_{m_j}(\ell_j)=0$ for all $j \in \mathbb{N}$. Consider the distribution $u \in \mathcal{D}'(\mathbb{S}^3)$ with Fourier-coefficients given by
$$ \widehat{u}(\ell)_{mn}:= \left\{ \begin{array}{ccc}\overline{\sigma_{-m}(\ell)}-\overline{q} & \textrm{if} & \ell=\ell_j, m=n=m_j \\ p & \textrm{if} & \ell=\ell_j, m=n=-m_j \\ 0 & & \textrm{otherwise}  \end{array}   \right. .$$
Like in the proof of Proposition \ref{Zinfinite}, it is easy to see that $Pu=0$. On the other hand, $u \notin C^\infty(\mathbb{S}^3)$ since $\widehat{u}(\ell_j)_{-m_j -m_j}=p$ for infinite indexes $j$, which guarantees that the sequence $(\|\widehat{u}(\ell_j)\|)_j$ does not decay. This concludes that $P$ is not (GH). 
\end{proof}

\begin{example} Consider numbers $r \in 2\cdot \mathbb{N}$, $p,q \in \mathbb{C}$ such that $\Re(q)=1$, $|q|=|p|$ and $a \in \mathbb{R}\setminus\{0\}$ with $(2/a)^{1/r} \in \frac{1}{2}\mathbb{Z}$. Let $L$ be the operator $L=a\cdot \partial_0^r$ acting on $\mathbb{S}^3$. In this case 
	\begin{eqnarray*} 
		|\Delta(\ell)_m| & = & |(am^r-q)\cdot (a(-m)^r-\overline{q})-|p|^2) \\
		& = & |a^2m^{2r}-2am^r\Re(q)+|q|^2-|p|^2|\\
		& = & |(am^r)^2-2am^r|.
	\end{eqnarray*}  
Since $(2/a)^{1/r}=m \in \frac{1}{2}\mathbb{Z}$, we have $am^r=2$, so $\Delta(\ell)_m=0$. By Proposition \ref{ZinfiniteSU2} the Vekua-type operator
$$Pu=a\partial_0^r u +pu+q\overline{u}$$
is not (GH). Note that since for all $\ell \in \mathbb{N}_0$ we have $0 \in J_\ell$, we have $\sigma_L(\ell)_0=0$, so the operator $L$ is also not (GH) (Theorem 3.2 of \cite{da2025diagonal}). 
\end{example} 

\begin{example} Let $G=\mathbb{S}^3 \times \mathbb{T}^1$ and consider numbers $a,q \in \mathbb{R}\setminus\{0\}$, $p \in \mathbb{C}\setminus \{0\}$ and $r \in \mathbb{N}$, such that $q \notin \frac{1}{2}\mathbb{Z}$ and also $q\pm |p| \notin \frac{1}{2}\mathbb{Z}$. Now consider the operator $L=\partial_0+ia D_t^r$ and the corresponding Vekua-type operator 
	$$Pu=\partial_0 u +iaD_t^ru+qu+p\overline{u},$$
	acting on $G$. In this case we have $\widehat{G} \equiv \frac{1}{2}\mathbb{N}_0 \times \mathbb{Z}$ and for each $(\ell,k) \in \widehat{G}$ and $m,n \in J_\ell$ we can see that $\sigma_L(\ell,k)_{mn}=\delta_{mn}\cdot (m+iak^r)$. So $L$ is a diagonal operator and

\begin{eqnarray*} \Delta(\ell,k)_m & = & (m+iak^r-q)^2-|p|^2   \\
	& = & m^2+2iamk^r-a^2k^{2r}-2q(m+iak^r)+q^2-|p|^2 \\
	& = & m^2-a^2k^{2r}-2qm+q^2-|p|^2+2iak^r(m-q)
\end{eqnarray*} 
If $k\neq 0$, then 
$$|\Delta(\ell,k)_m| \geq |2iak^r(m-q)| \geq 2|a||m-q|.$$
Since $q \notin \frac{1}{2}\mathbb{Z}$ and $\{|m-q|; m \in \frac{1}{2}\mathbb{Z}\}$ is discrete, there exists $C_1>0$ such that $|\Delta(\ell,k)_m| \geq C_1$.
If $k=0$, then 
$$\Delta(\ell,0)_m = m^2-2qm+q^2-|p|^2,$$
and 
$$\Delta(\ell,0)_m = 0 \Leftrightarrow m=\frac{2q\pm \sqrt{4q^2-4(q^2-|p|^2)}}{2}=q\pm |p|.$$
Again, since $q\pm |p| \notin \frac{1}{2}\mathbb{Z}$ and $\{|m^2 -2qm+q^2-|p|^2|; m \in \frac{1}{2}\mathbb{Z}\}$ is discrete, there exists $C_2>0$ such that $|\Delta(\ell,0)_m| \geq C_2$. In this way, taking $C:=\min\{C_1,C_2\}>0$, we have $|\Delta(\ell,k)_m| \geq C  \geq C \jp{(\ell,k)}^{-1}$ for all $(\ell,k) \in \widehat{G}$ and $m \in J_\ell$. By Proposition \ref{firstprop} the Vekua-type operator $P$ is (GH). On the other hand, for all $\ell \in \mathbb{N}_0$ we have $0 \in J_\ell$, so $\sigma_L(\ell,0)_0=0$, so by Theorem (3.2) of \cite{da2025diagonal}, the operator $L$ is not (GH) (on the another hand, $L$ is (GS) by Theorem 3.3 of the same reference). 

\end{example} 

\begin{example} A variation of the last example. Again consider the group $G=\mathbb{S}^3 \times \mathbb{T}^1$ and numbers $a,q \in \mathbb{R}\setminus\{0\}$, $p \in \mathbb{C}\setminus\{0\}$, but now suppose that $\Re(q)=0$ and $|p|=|q|$. Define the operator $L=\partial_0+iaD_t$ acting on $G$. In this case we have
\begin{eqnarray*}
 \Delta(\ell,k)_m & = & (m+iak-q)(m-iak-\overline{q})-|p|^2 \\
 & = & (m^2-a^2k^2)+2iamk.
\end{eqnarray*} 
In this way, for all $\ell \in \mathbb{N}_0$ we have $\Delta(\ell,0)_0=0$. By Proposition \ref{Zinfinite} the Vekua-type operator
$$Pu=\partial_0u+iaD_tu+qu+p\overline{u}$$
is not (GH). On the other hand, if $\Delta(\ell,k)_m \neq 0$, then $m,k \neq 0$, which implies $|m|\cdot |k|\geq \frac{1}{2}$, and then
$$|\Delta(\ell,k)_m| \geq |a|>0.$$
So by Proposition \ref{GSconstant} the operator $P$ is (GS). 
	
\end{example}

\section{A non-constant coefficient case}\label{nonconstant} 

Suppose that $D:\mathcal{D}'(G) \rightarrow \mathcal{D}'(G)$ is a diagonal operator in the sense of the last section, also suppose that $s,q:\mathbb{T}^1\rightarrow \mathbb{R}$ are smooth functions such that $q\geq 0$ and $q$ is not identically zero. Let also $p_0,\delta,\lambda \in \mathbb{R}$ and $\alpha \in \mathbb{C}\setminus \{0\}$. Consider the operator
\begin{equation}\label{Lnonconstant}
	L=\partial_t -(p_0+i\lambda q(t))\cdot D
\end{equation}
on $\mathbb{T}^1 \times G$ and the Vekua-type operator
\begin{equation}\label{Pnonconstant}
	Pu=Lu-(s(t)+i\delta q(t))\cdot u-\alpha q(t) \cdot \overline{u}
\end{equation}
on $\mathbb{T}^1\times G$. The main goal here is to present some sufficient conditions that guarantee $P$ is $C^\infty$-globally solvable, in the sense that for all $f \in C^\infty(\mathbb{T}^1 \times G)$ there exists $u \in C^\infty(\mathbb{T}^1 \times G)$ such that $Pu=f$. For this, define the following quantities:

\begin{itemize}
	\item $q_0=\int_0^{2\pi} q(\tau)d\tau$;
	\item $s_0=\int_0^{2\pi} s(\tau)d\tau$;
\end{itemize}

For each $[\xi]\in \widehat{G}$ we will write $\sigma_m(\xi)=a(\xi)_m+ib(\xi)_m$. In terms of the coefficients and the symbol, we are going to assume that:
\begin{itemize}
	\item[a)] $|\delta|\neq |\alpha|$;
	\item[b)] $|\lambda \sigma_m +\delta|\neq |\alpha|$ for all $[\xi]\in \widehat{G}$ and $1\leq m \leq d_\xi$. 
	We define $\rho_m(\xi)$ as the non-zero complex number such that $\Re(\rho_m(\xi))\geq 0$ and $\rho_m(\xi)^2 =|\alpha|^2-(\lambda \sigma_m(\xi)+\delta)^2$. 
	\item[c)] There are constants $C_0>0$ and $j_0 \in \mathbb{N}$ such that
	\begin{equation}\label{DC_rho} 
	\left | \rho_m(\xi) \right| \geq C_0 \jp{\xi}^{-j_0}
	\end{equation} 
	for all $[\xi] \in \widehat{G}$ and $1 \leq m \leq d_\xi$.
	\item[d)] If $p_0\neq 0$, we ask that there is constant $C>0$ such that 
	\begin{equation}\label{log} 
	|a_m(\xi)| \leq C \log \jp{\xi} 
	\end{equation}
	for all $\xi$ and $m$.
	\item[e)] Condition (DCn): there exists $M>0$ such that for all $[\xi]\in \widehat{G}$ with $\jp{\xi}>M$ the following inequalities hold:
	\begin{equation}\label{DCn}
		\left|  e^{-\rho_m(\xi)q_0}-e^{\pm(\sigma_m(\xi)2\pi p_0 +s_0)}\right|\geq \jp{\xi}^{-M}
	\end{equation}
	for all $1 \leq m \leq d_\xi$.
\end{itemize}

\begin{theorem}\label{mainthm} Under the hypothesis a), b), c), d) and e) from above, the operator $P$ defined in (\ref{Pnonconstant}) is globally solvable.
\end{theorem}

\begin{proof}
Suppose that $f \in C^\infty(\mathbb{T}^1 \times G)$ and that there exists $u \in \mathcal{D}'(\mathbb{T}^1 \times G)$ such that $Pu=f$ and fix $[\xi] \in \widehat{G}$. By taking the $[\xi]$-partial Fourier-coefficient on the equation $Pu=f$ we get
$$\partial_t \widehat{u}(t,\xi)-(p_0+i\lambda q(t))\sigma_D(\xi)\cdot \widehat{u}(t,\xi)-(s(t)+i\delta q(t))\widehat{u}(t,\xi)-\alpha q(t) \widehat{\overline{u}}(t,\xi)=\widehat{f}(t,\xi).$$
By taking the $mn$-entry of the above matrix, we use the hypothesis about $\sigma_D(\xi)$ to obtain
\begin{equation}
	\partial_t \widehat{u}(t,\xi)_{mn} -(p_0+i\lambda q(t))\sigma_m(\xi)\cdot \widehat{u}(t,\xi)_{mn}-(s(t)+i\delta q(t))\widehat{u}(t,\xi)_{mn}-\alpha q(t) \widehat{\overline{u}}(t,\xi)_{mn}=\widehat{f}(t,\xi)_{mn}
\end{equation}
Doing the analogue equation for the representation $\overline{\xi}$ and taking its complex conjugate we obtain
\begin{equation}
	\partial_t \widehat{\overline{u}}(t,\xi)_{mn}-(p_0-i\lambda q(t))\sigma_m(\xi)\cdot \widehat{\overline{u}}(t,\xi)_{mn}-(s(t)-i\delta q(t))\widehat{\overline{u}}(t,\xi)_{mn}-\overline{\alpha} q(t)\widehat{u}(t,\xi)_{mn}=\widehat{\overline{f}}(t,\xi)_{mn}. 
\end{equation}

Writing
$$w(\xi)_{mn}=\left[ \begin{array}{c} \widehat{u}(t,\xi)_{mn} \\ \widehat{\overline{u}}(t,\xi)_{mn} \end{array} \right], F(\xi)_{mn} =  \left[ \begin{array}{c} \widehat{f}(t,\xi)_{mn} \\ \widehat{\overline{f}}(t,\xi)_{mn} \end{array} \right],$$
the equations above mean that
\begin{equation}\label{edo} w(\xi)_{mn}' = M(\xi)_{mn} \cdot w(\xi)_{mn}+F(\xi)_{mn},
\end{equation}
where
$$M(\xi)_{mn}= \left[ \begin{array}{cc} (p_0+i\lambda q(t))\sigma_m(\xi)+(s(t)+i\delta q(t)) & \alpha q(t) \\ \overline{\alpha} q(t) & (p_0-i\lambda q(t))\sigma_m(\xi)+(s(t)-i\delta q(t)) \end{array}  \right] $$
Suppose that (\ref{edo}) has a smooth solution $w:\mathbb{T}^1 \rightarrow \mathbb{R}$ and define
$$y(\xi)_{mn}=e^{-\sigma_m(\xi)p_0 t-S(t)}\cdot w(\xi)_{mn},$$
where $S(t)=\int_0^t s(\tau)d\tau$. It follows from (\ref{edo}) that 
\begin{eqnarray*} y(\xi)_{mn}' & = & e^{-\sigma_m(\xi)p_0 t-S(t)} \cdot \left[M(\xi)_{mn}w(\xi)_{mn}+F(\xi)_{mn}\right] -\left[\sigma_m(\xi)p_0+s(t) \right]e^{-\sigma_m(\xi)p_0 t-S(t)}w(\xi)_{mn} \\
	& = & \left[ M(\xi)_{mn}-(\sigma_m(\xi)p_0+s(t))\cdot \textrm{Id}\right]y(\xi)_{mn}+e^{-\sigma_m(\xi)p_0 t-S(t)}\cdot F(\xi)_{mn}.
\end{eqnarray*}   
Note that
$$y(\xi)_{mn}(2\pi)=e^{-\sigma_m(\xi)p_0 2\pi -S(2\pi)}w(\xi)_{mn}(2\pi)=e^{-\sigma_m(\xi)p_0 2\pi-s_0}y(\xi)_{mn}(0),$$
so
$$y(\xi)_{mn}(0)=e^{\sigma_m(\xi)p_02\pi +s_0}\cdot y(\xi)_{mn}(2\pi).$$
But
$$M(\xi)_{mn}-(\sigma_m(\xi)p_0+s(t))\textrm{Id} =q(t) \cdot \widetilde{M}(\xi)_{mn},$$
where
$$\widetilde{M}(\xi)_{mn} = \left[ \begin{array}{cc} i(\lambda \sigma_m(\xi)+\delta) & \alpha \\ \overline{\alpha} & -i(\lambda \sigma_m(\xi)+\delta) \end{array}  \right] $$
The eigenvalues of $\widetilde{M}(\xi)_{mn}$ are $\pm \sqrt{|\alpha|^2-(\lambda \sigma_m(\xi)+\delta)^2}$. Condition b) guarantees that these eigenvalues are distinct. Recall that $\rho_m(\xi) \in \mathbb{C}$ is such that $\Re(\rho_m(\xi))\geq 0$ and $\rho_m(\xi)^2=|\alpha|^2-(\lambda \sigma_m(\xi)+\delta)^2$. If
$$V(\xi)_{mn}^+:= \left[ \begin{array}{c} \alpha \\ \rho_m(\xi)-i(\lambda \sigma_m(\xi)+\delta) \end{array} \right], V(\xi)_{mn}^-:= \left[ \begin{array}{c} \alpha \\ -\rho_m(\xi)-i(\lambda \sigma_m(\xi)+\delta) \end{array} \right],$$
then $V(\xi)_{mn}^+$ and $V(\xi)_{mn}^-$ are eigenvectors associated with $\rho_m(\xi)$ and $-\rho_m(\xi)$ respectively. 
If
$$T(\xi)_{mn} = \left[ \begin{array}{c|c} V(\xi)_{mn}^+ & V(\xi)_{mn}^- \end{array} \right],$$
then
\begin{equation} T(\xi)_{mn}^{-1} =\frac{1}{-2\alpha \rho_m(\xi)} \left[ \begin{array}{cr} -\rho_m(\xi)-i(\lambda \sigma_m(\xi)+\delta) & -\alpha \\ -\rho_m(\xi)+i(\lambda \sigma_m(\xi)+\delta) & \alpha \end{array} \right]
\end{equation}
and
\begin{equation}T(\xi)_{mn}^{-1}\widetilde{M}(\xi)_{mn}T(\xi)_{mn} = \rho_m(\xi) \cdot \left[ \begin{array}{cr} 1 & 0 \\ 0 & -1 \end{array}\right].\end{equation}

Coming back to the expression of $y(\xi)_{mn}'$ we get
$$y'(\xi)_{mn}'=q(t)\rho_m(\xi)T(\xi)_{mn} \cdot \left[ \begin{array}{cr} 1 & 0 \\ 0 & -1 \end{array}  \right] T(\xi)_{mn}^{-1} \cdot y(\xi)_{mn}+e^{-\sigma_m(\xi)p_0 t-S(t)}T(\xi)_{mn}T(\xi)_{mn}^{-1} F(\xi)_{mn}.$$
Defining $z(\xi)_{mn}=T(\xi)_{mn}^{-1} \cdot y(\xi)_{mn}$ and $G(\xi)_{mn}=T(\xi)_{mn}^{-1}F(\xi)_{mn}$, we get
\begin{eqnarray*} z(\xi)_{mn}'& = & T(\xi)_{mn}^{-1}\cdot y(\xi)_{mn}'\\
	& = & T(\xi)_{mn}^{-1} \left[q(t)\rho_m(\xi) T(\xi)_{mn} \left[ \begin{array}{cr} 1 & 0 \\ 0 & -1 \end{array}  \right] z(\xi)_{mn}+e^{-\sigma_m(\xi)p_0 t-S(t)}T(\xi)_{mn} G(\xi)_{mn}  \right] \\
	& = & q(t)\rho_m(\xi) \left[ \begin{array}{cr} 1 & 0 \\ 0 & -1 \end{array}  \right] \cdot z(\xi)_{mn}+e^{-\sigma_m(\xi)p_0 t-S(t)}G(\xi)_{mn} 
\end{eqnarray*} 
Moreover,
\begin{eqnarray}\label{znobordo} z(\xi)_{mn}(0) & = & T(\xi)_{mn}^{-1}\cdot y(\xi)_{mn}(0) \nonumber \\
	& = & e^{\sigma_m(\xi)p_0 2\pi+s_0}T(\xi)_{mn}^{-1}\cdot y(\xi)_{mn}(2\pi) \nonumber \\
	& = & e^{\sigma_m(\xi)p_0 2\pi +s_0}\cdot z(\xi)_{mn}(2\pi).
\end{eqnarray}
Component-wise we have
$$z_1(\xi)_{mn}'=\rho_m(\xi)q(t)z_1(\xi)_{mn}+e^{-\sigma_m(\xi)p_0 t-S(t)}G_1(\xi)_{mn}$$
and
$$z_2(\xi)_{mn}'=-\rho_m(\xi)q(t)z_2(\xi)_{mn}+e^{-\sigma_m(\xi)p_0 t-S(t)}G_2(\xi)_{mn}.$$
Its solutions are given by
\begin{eqnarray*}
	z_1(\xi)_{mn}& = & e^{\rho_m(\xi)\widetilde{Q}(t)}\left[ z_1(\xi)_{mn}(2\pi)+\int_{2\pi}^t e^{-\rho_m(\xi)\widetilde{Q}(\tau)}e^{-\sigma_m(\xi)p_0 \tau -S(\tau)}\cdot G_1(\xi)_{mn}d\tau \right] \\
	& = & -e^{\rho_m(\xi)\widetilde{Q}(t)}\int_t^{2\pi}e^{-\rho_m(\xi)\widetilde{Q}(\tau)}e^{-\sigma_m(\xi)p_0 \tau -S(\tau)}G_1(\xi)_{mn}d\tau+e^{\rho_m(\xi)\widetilde{Q}(t)}z_1(\xi)_{mn}(2\pi)
\end{eqnarray*}
where $\widetilde{Q}(t)=\int_{2\pi}^t q(\tau)d\tau = -\int_{t}^{2\pi} q(\tau)d\tau$, and
\begin{eqnarray*}
		z_2(\xi)_{mn}& = & e^{-\rho_m(\xi)Q(t)}\left[ z_2(\xi)_{mn}(0)+\int_0^t e^{\rho_m(\xi)Q(\tau)}e^{-\sigma_m(\xi)p_0 \tau -S(\tau)}\cdot G_2(\xi)_{mn}d\tau \right] \\
		& = & e^{-\rho_m(\xi)Q(t)}\int_0^t e^{\rho_m(\xi)Q(\tau)}e^{-\sigma_m(\xi)p_0 \tau -S(\tau)}G_2(\xi)_{mn}d\tau +e^{-\rho_m(\xi)Q(t)}z_2(\xi)_{mn}(0),
\end{eqnarray*}
where $Q(t)=\int_0^t q(\tau)d\tau$. Note that $Q$ and $\widetilde{Q}$ satisfy $\widetilde{Q}(0)=-q_0$, $\widetilde{Q}(2\pi)=0$, $Q(0)=0$ and $Q(2\pi)=q_0$. In particular, 
\begin{eqnarray*} z_1(\xi)_{mn}(0) & = & -e^{-\rho_m(\xi)q_0}\int_0^{2\pi}e^{-\rho_m(\xi)\widetilde{Q}(\tau)}e^{-\sigma_m(\xi)p_0 \tau -S(\tau)}G_1(\xi)_{mn}d\tau +e^{-\rho_m(\xi)q_0}z_1(\xi)_{mn}(2\pi)\\
	& = & -\int_0^{2\pi} e^{-\rho_m(\xi)(q_0+\widetilde{Q}(\tau))}e^{-\sigma_m(\xi)p_0 \tau-S(\tau)}G_1(\xi)_{mn}d\tau +e^{-\rho_m(\xi)q_0}z_1(\xi)_{mn}(2\pi).
\end{eqnarray*} 
By relation (\ref{znobordo}), we have
\begin{equation}\label{z_1nobordo}
	\left( e^{-\rho_m(\xi)q_0}-e^{\sigma_m(\xi)p_0 2\pi+s_0}\right)\cdot z_1(\xi)_{mn}(2\pi) = \int_0^{2\pi} e^{-\rho_m(\xi)(q_0+\widetilde{Q}(\tau))}e^{-\sigma_m(\xi)p_0 \tau-S(\tau)}G_1(\xi)_{mn}d\tau.
\end{equation}
Similarly, we have
$$z_2(\xi)_{mn}(2\pi)=e^{-\rho_m(\xi)q_0}\int_0^{2\pi} e^{\rho_m(\xi)Q(\tau)}e^{-\sigma_m(\xi)p_0 \tau -S(\tau)}G_2(\xi)_{mn}d\tau +e^{-\rho_m(\xi)q_0}z_2(\xi)_{mn}(0)$$
and again by using the relation (\ref{znobordo}) we get
$$z_2(\xi)_{mn}(2\pi)-e^{-\rho_m(\xi)q_0}e^{\sigma_m(\xi)p_0 2\pi+s_0}z(\xi)_{mn}(2\pi)=\int_0^{2\pi} e^{\rho_m(\xi)(-q_0+Q(\tau))}e^{-\sigma_m(\xi)p_0 \tau -S(\tau)}G_2(\xi)_{mn}d\tau,$$
so
\begin{equation}\label{z_2nobordo}
\left(1-e^{-\rho_m(\xi)q_0+\sigma_m(\xi)p_0 2\pi+s_0}  \right)\cdot z_2(\xi)_{mn}(2\pi) = \int_0^{2\pi} e^{\rho_m(\xi)(-q_0+Q(\tau))}e^{-\sigma_m(\xi)p_0 \tau -S(\tau)}G_2(\xi)_{mn}d\tau.
\end{equation}

Condition e) guarantees that the coefficients of $z_1(\xi)_{mn}(2\pi)$ and $z_2(\xi)_{mn}(2\pi)$ in the expressions above do not vanish.

 In this way, the solution $z(\xi)_{mn}$ is given by
 \begin{align}\label{solz_1} z_1(\xi)_{mn} = & -\int_t^{2\pi}e^{\rho_m(\xi)(\widetilde{Q}(t)-\widetilde{Q}(\tau))}e^{-\sigma_m(\xi)p_0 \tau -S(\tau)}G_1(\xi)_{mn}d\tau \nonumber \\
 	+ & e^{\rho_m(\xi)\widetilde{Q}(t)} \int_0^{2\pi} \frac{ e^{-\rho_m(\xi)(q_0+\widetilde{Q}(\tau))}e^{-\sigma_m(\xi)p_0 \tau-S(\tau)}G_1(\xi)_{mn}}{ e^{-\rho_m(\xi)q_0}-e^{\sigma_m(\xi)2\pi p_0 +s_0}}d\tau
 \end{align}
 and
 \begin{align}\label{solz_2}
 	z_2(\xi)_{mn} =& \int_0^t e^{\rho_m(\xi)(Q(\tau)-Q(t))}e^{-\sigma_m(\xi)p_0 \tau -S(\tau)}G_2(\xi)_{mn}d\tau \nonumber \\
 	+ & e^{-\rho_m(\xi)Q(t)} \int_0^{2\pi} \frac{e^{\rho_m(\xi)(-q_0+Q(\tau))}e^{-\sigma_m(\xi)p_0 \tau -S(\tau)}G_2(\xi)_{mn}}{e^{-\sigma_m 2\pi p_0-s_0}-e^{-\rho_m q_0}}d\tau
 \end{align}
 
 Now we must estimate the original solution. Recall that from the definitions of $z(\xi)_{mn}$ and $y(\xi)_{mn}$ we have
 \begin{eqnarray*} w(\xi)_{mn} & = & e^{\sigma_m(\xi)p_0 t+S(t)}\cdot y(\xi)_{mn}\\
   & = & e^{\sigma_m(\xi)p_0 t +S(t)} \cdot T(\xi)_{mn} \cdot z(\xi)_{mn}
 \end{eqnarray*} 
 so
 \begin{equation}\label{solutionu}
 	\widehat{u}(t,\xi)_{mn}  = \alpha e^{\sigma_m(\xi)p_0 t +S(t)}\cdot (z_1(\xi)_{mn}+z_2(\xi)_{mn}),
 \end{equation} 
 and then for all $L \in \mathbb{N}_0$ we have
 $$\frac{d^L}{dt^L} \widehat{u}(t,\xi)_{mn} = \alpha \sum_{k=0}^L \binom{L}{k} \frac{d^{L-k}}{dt^{L-k}} e^{\sigma_m(\xi) p_0 t +S(t)} \frac{d^k}{dt^k}(z_1(\xi)_{mn}+z_2(\xi)_{mn}).$$  
 For each $k=0,...,L$ we have
 \begin{eqnarray*} \frac{d^{L-k}}{dt^{L-k}} e^{\sigma_m(\xi) p_0 t +S(t)} & = & \sum_{n=0}^{L-k} \binom{L-k}{n} \frac{d^n}{dt^n} e^{S(t)} \frac{d^{L-k-n}}{dt^{L-k-n}} e^{\sigma_m(\xi) p_0 t} \\
 	& = & \sum_{n=0}^{L-k} \binom{L-k}{n} (\sigma_m(\xi) p_0)^{L-k-n} \frac{d^n}{dt^n} e^{S(t)}.
 \end{eqnarray*} 
 By Faà di Bruno formula, for all $n \in \mathbb{N}$ we have
 $$\frac{d^n}{dt^n} e^{S(t)} = e^{S(t)} \sum_{\gamma \in \Delta(n)} \frac{n!}{\gamma!} \prod_{\ell=1}^n \left( \frac{1}{\ell!} \frac{d^\ell}{dt^\ell} S(t)\right)^{\gamma_\ell},$$
 so for all $t \in [0,2\pi]$ we have
 \begin{eqnarray*}
 	\left| \frac{d^{L-k}}{dt^{L-k}} e^{\sigma_m(\xi)p_0 t+S(t)} \right| & = & \left| \sum_{n=0}^{L-k} \binom{L-k}{n} (\sigma_m(\xi) p_0)^{L-k-n} e^{S(t)} \sum_{\gamma \in \Delta(n)} \frac{n!}{\gamma!} \prod_{\ell=1}^n \left( \frac{1}{\ell!} \frac{d^\ell}{dt^\ell} S(t)\right)^{\gamma_\ell} \right|  \\ 
	& \leq & e^{S(t)} \sum_{n=0}^{L-k} \binom{L-k}{n} |\sigma_m(\xi)|^{L-k-n} |p_0|^{L-k-n} \sum_{\gamma \in \Delta(n)} \frac{n!}{\gamma!} \prod_{\ell=1}^n \left| \frac{1}{\ell!} \frac{d^\ell}{dt^\ell} S(t)\right|^{\gamma_\ell}
 \end{eqnarray*}
 Taking the supremum in $t \in [0,2\pi]$ on the inequality above and using (\ref{condsimbolo}), it follows that there is a constant $C>0$ depending only on $L$, $k$ and the coefficients of the operator such that
 \begin{equation}\label{exp_simbolo}
 \left| \frac{d^{L-k}}{dt^{L-k}} e^{\sigma_m(\xi)p_0 t+S(t)} \right| \leq C\jp{\xi}^{K(L-k)}
 \end{equation} 
 
 From $G(\xi)_{mn}=T(\xi)_{mn}^{-1} \cdot F(\xi)_{mn}$ we get
 $$G(\xi)_{mn}=\frac{1}{-2\alpha \rho_m(\xi)} \left[ \begin{array}{cr} -\rho_m(\xi)-i(\lambda \sigma_m(\xi)+\delta) & -\alpha \\ -\rho_m(\xi)+i(\lambda \sigma_m(\xi)+\delta) & \alpha \end{array} \right] \cdot \left[ \begin{array}{c} \widehat{f}(t,\xi)_{mn} \\ \widehat{\overline{f}}(t,\xi)_{mn} \end{array} \right].$$
 Since $1/|\rho_m|$ (assumption c) ), $|\rho_m|$ and $|\sigma_m|$ are of at most polynomial growth and $f$ is smooth, we conclude that $G(\xi)_{mn}$ is smooth. So it is enough to prove that $z_1(\xi)_{mn}$ and $z_2(\xi)_{mn}$ are rapidly decreasing. Since their expressions are similar, we will deal only with $z_1(\xi)_{mn}$, since the estimates for $z_2(\xi)_{mn}$ are totally analogous. We will start by estimating the term
 $$-e^{\rho_m \widetilde{Q}(t)}\int_t^{2\pi}e^{-\rho_m(\xi) \widetilde{Q}(\tau))}e^{-\sigma_m(\xi)p_0 \tau -S(\tau)}G_1(\xi)_{mn}d\tau.$$
 Fixing $L \in \mathbb{N}_0$, for each $k=0,1...,L$ we have
 \begin{eqnarray*}
 	& & \frac{d^k}{dt^k} \left(-e^{\rho_m \widetilde{Q}(t)}\int_t^{2\pi}e^{-\rho_m(\xi) \widetilde{Q}(\tau)}e^{-\sigma_m(\xi)p_0 \tau -S(\tau)}G_1(\xi)_{mn}d\tau \right) \\
 	& = & -\sum_{n=0}^k \binom{k}{n} \frac{d^{k-n}}{dt^{k-n}} e^{\rho_m \widetilde{Q}(t)} \frac{d^n}{dt^n} \left( \int_t^{2\pi}e^{-\rho_m(\xi) \widetilde{Q}(\tau)}e^{-\sigma_m(\xi)p_0 \tau -S(\tau)}G_1(\xi)_{mn}d\tau \right) \\
 	& = & - \left( \frac{d^k}{dt^k} e^{\rho_m\widetilde{Q}(t)} \right) \int_t^{2\pi}e^{-\rho_m(\xi) \widetilde{Q}(\tau)}e^{-\sigma_m(\xi)p_0 \tau -S(\tau)}G_1(\xi)_{mn}d\tau \\
 	& + &\sum_{k=1}^n \binom{k}{n} \frac{d^{k-n}}{dt^{k-n}} e^{\rho_m \widetilde{Q}(t)} \frac{d^{n-1}}{dt^{n-1}} \left(e^{-\rho_m(\xi) \widetilde{Q}(t)}e^{-\sigma_m(\xi)p_0 t -S(t)}G_1(\xi)_{mn} \right)  
 \end{eqnarray*}
 Again by Faà di Bruno formula, for all $N \in \mathbb{N}$ we have
 \begin{align}\label{bruno_1}
 	\frac{d^N}{dt^N} e^{\rho_m \widetilde{Q}(t)} & =  e^{\rho_m \widetilde{Q}(t)} \sum_{\gamma \in \Delta(N)} \frac{N!}{\gamma!} \prod_{\ell=1}^N \left(\frac{1}{\ell!} \frac{d^\ell}{dt^{\ell}} \rho_m \widetilde{Q}(t) \right)^{\gamma_\ell} \nonumber \\
 	 & = e^{\rho_m \widetilde{Q}(t)} \sum_{\gamma \in \Delta(N)} \rho_m^{\gamma_1+...+\gamma_N} \frac{N!}{\gamma!}\prod_{\ell=1}^N \left(\frac{1}{\ell!} \frac{d^\ell}{dt^{\ell}} \widetilde{Q}(t) \right)^{\gamma_\ell}
 \end{align}
 Since $\gamma_1+...+\gamma_N \leq N$ for all $\gamma \in \Delta(N)$, we get
 \begin{align}\label{bruno_2}
 	\left| 	\frac{d^N}{dt^N} e^{\rho_m \widetilde{Q}(t)} \right| & \leq e^{\Re(\rho_m) \widetilde{Q}(t)} \sum_{\gamma \in \Delta(N)} |\rho_m|^N \frac{N!}{\gamma!}\prod_{\ell=1}^N \frac{1}{\ell!}\left| \frac{d^\ell}{dt^{\ell}} \widetilde{Q}(t) \right|^{\gamma_\ell} \nonumber \\
 	& \leq C e^{\Re(\rho_m)\widetilde{Q}(t)} \jp{\xi}^{NK} 
 \end{align}
 where $C>0$ is a constant that depends only on the operator and $N$. From (\ref{bruno_1}) we get
 \begin{eqnarray*}
  & - & \left( \frac{d^k}{dt^k} e^{\rho_m\widetilde{Q}(t)} \right) \int_t^{2\pi}e^{-\rho_m(\xi) \widetilde{Q}(\tau)}e^{-\sigma_m(\xi)p_0 \tau -S(\tau)}G_1(\xi)_{mn}d\tau\\
  & = & -\sum_{\gamma \in \Delta(k)} \frac{k!}{\gamma!} \prod_{\ell=1}^k \left( \frac{1}{\ell!} \frac{d^\ell}{dt^{\ell}} \widetilde{Q}(t)\right)^{\gamma_\ell} \rho_m^{\gamma_1+...+\gamma_k} \int_t^{2\pi} e^{\rho_m(\widetilde{Q}(t)-\widetilde{Q}(\tau))}e^{-\sigma_m(\xi)p_0 \tau-S(\tau)}G_1(\xi)_{mn}d\tau  
 \end{eqnarray*}
 Also notice that since we are supposing $q\geq 0$, for $\tau \in [t,2\pi]$ we have
 \begin{eqnarray*}
 	\widehat{Q}(t)-\widetilde{Q}(\tau) & = & -\int_t^{2\pi} q+\int_\tau^{2\pi} q\\
 	& = & -\int_t^\tau q -\int_\tau^{2\pi}q + \int_\tau^{2\pi} q \\
 	& = & -\int_t^\tau q \leq 0
 \end{eqnarray*}
 and since $\Re(\rho_m)\geq 0$, we get $e^{\Re(\rho_m)(\widetilde{Q}(t)-\widetilde{Q}(\tau))}\leq 1 $, so there is a constant $\widetilde{C}>0$ which depends only on $k$ and the operator such that
 
 \begin{equation} \left| \left( \frac{d^k}{dt^k} e^{\rho_m\widetilde{Q}(t)} \right) \int_t^{2\pi}e^{-\rho_m(\xi) \widetilde{Q}(\tau)}e^{-\sigma_m(\xi)p_0 \tau -S(\tau)}G_1(\xi)_{mn}d\tau \right|
   \leq \widetilde{C}  \jp{\xi}^{k\cdot K} e^{2 \pi|p_0|\cdot |a_m| } \sup_{t \in [0,2\pi]} |G_1(\xi)_{mn}|.  
 \end{equation}
 By condition (\ref{log}) and the fact that $G_1(\xi)_{mn}$ is smooth, we conclude that the term
 $$\left| \left( \frac{d^k}{dt^k} e^{\rho_m\widetilde{Q}(t)} \right) \int_t^{2\pi}e^{-\rho_m(\xi) \widetilde{Q}(\tau)}e^{-\sigma_m(\xi)p_0 \tau -S(\tau)}G_1(\xi)_{mn}d\tau \right|$$
 is rapidly decreasing. Now
 \begin{align*}
 	& \frac{d^{n-1}}{dt^{n-1}} \left(e^{-\rho_m(\xi) \widetilde{Q}(t)}e^{-\sigma_m(\xi)p_0 t -S(t)}G_1(\xi)_{mn} \right)   \\
 	& =  \sum_{r=0}^{n-1}\binom{n-1}{r} \frac{d^r}{dt^r}  e^{-\rho_m(\xi) \widetilde{Q}(t)}e^{-\sigma_m(\xi)p_0 t -S(t)} \frac{d^{n-1-r}}{dt^{n-1-r}}G_1(\xi)_{mn}  \\
    & = \sum_{r=0}^{n-1} \binom{n-1}{r} \sum_{w=0}^r \binom{r}{w} \frac{d^w}{dt^w} e^{-\sigma_m p_0 t-S(t)} \frac{d^{r-w}}{dt^{r-w}} e^{-\rho_m \widetilde{Q}(t)} \frac{d^{n-1-r}}{dt^{n-1-r}}G_1(\xi)_{mn}. 
 \end{align*}
 so by inequalities (\ref{exp_simbolo}) and (\ref{bruno_2}) we have
 \begin{align}
 	& \left| \frac{d^{n-1}}{dt^{n-1}} \left(e^{-\rho_m(\xi) \widetilde{Q}(t)}e^{-\sigma_m(\xi)p_0 t -S(t)}G_1(\xi)_{mn} \right) \right| \nonumber \\
 	& \leq \sum_{r=0}^{n-1} \binom{n-1}{r} \sum_{w=0}^r \binom{r}{w} \left|\frac{d^w}{dt^w} e^{-\sigma_m p_0 t-S(t)} \right| \left| \frac{d^{r-w}}{dt^{r-w}} e^{-\rho_m \widetilde{Q}(t)}\right| \left|\frac{d^{n-1-r}}{dt^{n-1-r}}G_1(\xi)_{mn} \right| \nonumber \\
 	& \leq \widetilde{C} \sum_{r=0}^{n-1} \binom{n-1}{r} \sum_{w=0}^r \binom{r}{w} \jp{\xi}^{Kw} e^{-\Re(\rho_m)\widetilde{Q}(t)} \jp{\xi}^{K(r-w)} \sup_{\substack{t \in [0,2\pi] \\ m\leq n-1}} \left|\frac{d^{m}}{dt^{m}}G_1(\xi)_{mn} \right| \nonumber \\
 	& \leq C e^{-\Re(\rho_m)\widetilde{Q}(t)} \jp{\xi}^{K(n-1)}\sup_{\substack{t \in [0,2\pi] \\ m\leq n-1}} \left|\frac{d^{m}}{dt^{m}}G_1(\xi)_{mn} \right|,
 \end{align}
 and we conclude that
 \begin{align}
 	& \left| \sum_{k=1}^n \binom{k}{n} \frac{d^{k-n}}{dt^{k-n}} e^{\rho_m \widetilde{Q}(t)} \frac{d^{n-1}}{dt^{n-1}} \left(e^{-\rho_m(\xi) \widetilde{Q}(t)}e^{-\sigma_m(\xi)p_0 t -S(t)}G_1(\xi)_{mn} \right) \right| \nonumber \\
 	& \leq \sum_{k=1}^n \binom{k}{n} \left| \frac{d^{k-n}}{dt^{k-n}} e^{\rho_m \widetilde{Q}(t)}\right| \left|  \frac{d^{n-1}}{dt^{n-1}} \left(e^{-\rho_m(\xi) \widetilde{Q}(t)}e^{-\sigma_m(\xi)p_0 t -S(t)}G_1(\xi)_{mn} \right) \right| \nonumber \\
 	& \leq C \sum_{k=1}^n \binom{k}{n} e^{\Re(\rho_m) \widetilde{Q}(t)} \jp{\xi}^{(k-n)K} e^{-\Re(\rho_m)\widetilde{Q}(t)} \jp{\xi}^{K(n-1)}\sup_{\substack{t \in [0,2\pi] \\ m\leq n-1}} \left|\frac{d^{m}}{dt^{m}}G_1(\xi)_{mn} \right| \nonumber \\
    & \leq C \jp{\xi}^{K(k-1)} \sup_{\substack{t \in [0,2\pi] \\ m\leq n-1}} \left|\frac{d^{m}}{dt^{m}}G_1(\xi)_{mn} \right|
 \end{align}
 where $C>0$ is a constant that depends only on $n$ and the operator. Again since $G_1(\xi)_{mn}$ is smooth, we conclude that 
 $$\left| -e^{\rho_m \widetilde{Q}(t)}\int_t^{2\pi}e^{-\rho_m(\xi) \widetilde{Q}(\tau))}e^{-\sigma_m(\xi)p_0 \tau -S(\tau)}G_1(\xi)_{mn}d\tau \right|$$
 is rapidly decreasing. Finally, we need to estimate the term
 \begin{equation}\label{lastterm}
 e^{\rho_m(\xi)\widetilde{Q}(t)} \int_0^{2\pi} \frac{ e^{-\rho_m(\xi)(q_0+\widetilde{Q}(\tau))}e^{-\sigma_m(\xi)p_0 \tau-S(\tau)}G_1(\xi)_{mn}}{ e^{-\rho_m(\xi)q_0}-e^{\sigma_m(\xi)2\pi p_0 +s_0}}d\tau.
 \end{equation} 
 Since $q \geq 0$ it is easy to see that $q_0+\widetilde{Q}(\tau)\geq 0$ for all $ \tau \in [0,2\pi]$. Combining this fact with (\ref{bruno_2}), conditions (\ref{log}) and (\ref{DCn}),  and the fact that $G_1(\xi)_{mn}$ is smooth, it is clear the term (\ref{lastterm}) is rapidly decreasing. Similar estimates can be done to prove that $z_2(\xi)_{mn}$ is also rapidly decreasing, so the sequence $\widehat{u}(\xi)_{mn}$ indeed define a smooth function $u$ such that $Pu=f$.
 \end{proof} 
 
 Now we want to explore some consequences of Theorem \ref{mainthm}. Conditions c) and e) are more difficult to deal with since is not easy to estimate those expressions in general. The ideia is to find classes of examples such that, for example, $|\rho_m(\xi)|$ and $|e^{-\rho_m(\xi)q_0}-e^{\pm(\sigma_m(\xi)2\pi p_0+s_0)}|$ are bounded from below by a positive constant. 
 
 \begin{example} Suppose that $\lambda=0$, $|\alpha|>\delta>0$, $p_0>0$ and fix a compact Lie group $H$ such that $\widehat{H}=\{[\xi_j]; j \in \mathbb{N}\}$ and $[\xi_j]=[\overline{\xi_j}]$ for all $j \in \mathbb{N}$ (for example, take $H=(\mathbb{S}^3)^r$) and choose an increasing sequence $(a_j)_j$ of positive real numbers such that $|a_j| \leq \log \jp{\xi_j}$ for all $j\in \mathbb{N}$ and $\lim_j a_j=\infty$. Let $P: \mathcal{D}'(H) \rightarrow \mathcal{D}'(H)$ the left-invariant continuous operator such that $\sigma_P(\xi_j) = a_j \cdot \textrm{Id}_{d_{\xi_j}}$ for all $j \in \mathbb{N}$. Now consider the compact Lie group $G=H \times \mathbb{T}^1$ and the operator $D=P+i \cdot D_{t_2}^r$, $r \in \mathbb{N}$, acting on $G$ with variables $(x,t_2)$. In this case we have $\widehat{G} \equiv \widehat{H} \times \mathbb{Z}$ and it is easy to see that for each $(j,k) \in \widehat{G}$ we have $\sigma_D(\xi_j,k)=(a_j+ik)\cdot \textrm{Id}_{d_{\xi_j}}$, so $D$ is a diagonal operator. Also, since $\lambda=0$ and $|\alpha|>\delta$, we have $\rho=\sqrt{|\alpha|^2-\delta^2}>0$. Additionally, we assume that $a_j \neq (-s_0\pm \rho q_0)/2\pi p_0$ for all $j$. In this way, there exists $C>0$ such that 
 	$$|e^{-\rho q_0} -e^{\pm(2\pi p_0 a_j+s_0)}| \geq C$$
 	for all $j$. By Theorem \ref{mainthm} the Vekua-type operator $$P=\partial_t-p_0(P+iD_{t_2}^r)-(s(t)+i\delta q(t))u-\alpha q(t)\overline{u}$$ 
 is globally solvable on $\mathbb{T}^2 \times H$.
 	
 \end{example} 
 
 \begin{example} Suppose that $\lambda =0$ and $|\alpha|>\delta>0$ and $p_0 >0$. Let $G$ be a compact Lie group such that $\widehat{G}=\{[\xi_j]; j \in \mathbb{N}\}\cup \{[\overline{\xi_j}]; j \in \mathbb{N}\}$, with $[\xi_n]\neq [\overline{\xi_n}]$ for all $n$ and $[\xi_n] \neq [\xi_m]$ for all $n \neq m$. Let $(c_n)_n$ be a sequence of complex numbers such that there exists $C>0$ and $r \in \mathbb{N}$ with $|\Re(c_n)|\leq \log \jp{\xi_n}$ and $|c_n| \leq C \jp{\xi_n}^r$ for all $n\in \mathbb{N}$. We also suppose that $(\Re(c_n))_n$ is an increasing sequence of positive numbers such that $\lim_n \Re(c_n)=\infty$. Consider the left-invariant continuous operator $D=\mathcal{D}'(G) \rightarrow \mathcal{D}'(G)$ such that $\sigma_D(\xi_n)=c_n \cdot \textrm{Id}_{d_{\xi_n}}$ and $\sigma_D(\overline{\xi_n}) = \overline{c_n}\cdot \textrm{Id}_{d_{\xi_n}}$ for all $n\in \mathbb{N}$. Then $D$ is a diagonal operator. In this case we have $\rho=\sqrt{|\alpha|^2-\delta^2}>0$. Additionally, suppose that $a_n \neq (-s_0\pm \rho q_0)/2\pi p_0$ for all $n$. Then, as in the last example, there exists $\widehat{C}>0$ such that
 	$$|e^{-\rho q_0} - e^{\pm( 2\pi p_0 c_n+s_0)}| \geq \widehat{C}$$
 for all $n\in \mathbb{N}$. By Theorem \ref{mainthm} the operator
 $$P=\partial_t -p_0 D -(s(t)+i\delta q(t))u - \alpha q(t)\overline{u}$$
 is globally solvable on $\mathbb{T}^1 \times G$.
 	
 \end{example}

\bibliography{references}


\begin{thebibliography}{19}
\ifx \bisbn   \undefined \def \bisbn  #1{ISBN #1}\fi
\ifx \binits  \undefined \def \binits#1{#1}\fi
\ifx \bauthor  \undefined \def \bauthor#1{#1}\fi
\ifx \batitle  \undefined \def \batitle#1{#1}\fi
\ifx \bjtitle  \undefined \def \bjtitle#1{#1}\fi
\ifx \bvolume  \undefined \def \bvolume#1{\textbf{#1}}\fi
\ifx \byear  \undefined \def \byear#1{#1}\fi
\ifx \bissue  \undefined \def \bissue#1{#1}\fi
\ifx \bfpage  \undefined \def \bfpage#1{#1}\fi
\ifx \blpage  \undefined \def \blpage #1{#1}\fi
\ifx \burl  \undefined \def \burl#1{\textsf{#1}}\fi
\ifx \doiurl  \undefined \def \doiurl#1{\url{https://doi.org/#1}}\fi
\ifx \betal  \undefined \def \betal{\textit{et al.}}\fi
\ifx \binstitute  \undefined \def \binstitute#1{#1}\fi
\ifx \binstitutionaled  \undefined \def \binstitutionaled#1{#1}\fi
\ifx \bctitle  \undefined \def \bctitle#1{#1}\fi
\ifx \beditor  \undefined \def \beditor#1{#1}\fi
\ifx \bpublisher  \undefined \def \bpublisher#1{#1}\fi
\ifx \bbtitle  \undefined \def \bbtitle#1{#1}\fi
\ifx \bedition  \undefined \def \bedition#1{#1}\fi
\ifx \bseriesno  \undefined \def \bseriesno#1{#1}\fi
\ifx \blocation  \undefined \def \blocation#1{#1}\fi
\ifx \bsertitle  \undefined \def \bsertitle#1{#1}\fi
\ifx \bsnm \undefined \def \bsnm#1{#1}\fi
\ifx \bsuffix \undefined \def \bsuffix#1{#1}\fi
\ifx \bparticle \undefined \def \bparticle#1{#1}\fi
\ifx \barticle \undefined \def \barticle#1{#1}\fi
\bibcommenthead
\ifx \bconfdate \undefined \def \bconfdate #1{#1}\fi
\ifx \botherref \undefined \def \botherref #1{#1}\fi
\ifx \url \undefined \def \url#1{\textsf{#1}}\fi
\ifx \bchapter \undefined \def \bchapter#1{#1}\fi
\ifx \bbook \undefined \def \bbook#1{#1}\fi
\ifx \bcomment \undefined \def \bcomment#1{#1}\fi
\ifx \oauthor \undefined \def \oauthor#1{#1}\fi
\ifx \citeauthoryear \undefined \def \citeauthoryear#1{#1}\fi
\ifx \endbibitem  \undefined \def \endbibitem {}\fi
\ifx \bconflocation  \undefined \def \bconflocation#1{#1}\fi
\ifx \arxivurl  \undefined \def \arxivurl#1{\textsf{#1}}\fi
\csname PreBibitemsHook\endcsname

\bibitem[\protect\citeauthoryear{Greenfield and Wallach}{1972}]{GW1972_pams}
\begin{barticle}
\bauthor{\bsnm{Greenfield}, \binits{S.J.}},
\bauthor{\bsnm{Wallach}, \binits{N.R.}}:
\batitle{Global hypoellipticity and {Liouville} numbers}.
\bjtitle{Proc. Am. Math. Soc.}
\bvolume{31},
\bfpage{112}--\blpage{114}
(\byear{1972})
\doiurl{10.2307/2038523}
\end{barticle}
\endbibitem

\bibitem[\protect\citeauthoryear{Greenfield and Wallach}{1973a}]{GW1973_tams}
\begin{barticle}
\bauthor{\bsnm{Greenfield}, \binits{S.J.}},
\bauthor{\bsnm{Wallach}, \binits{N.R.}}:
\batitle{Remarks on global hypoellipticity}.
\bjtitle{Trans. Am. Math. Soc.}
\bvolume{183},
\bfpage{153}--\blpage{164}
(\byear{1973})
\doiurl{10.2307/1996464}
\end{barticle}
\endbibitem

\bibitem[\protect\citeauthoryear{Greenfield and
  Wallach}{1973b}]{greenfield1973globally}
\begin{barticle}
\bauthor{\bsnm{Greenfield}, \binits{S.J.}},
\bauthor{\bsnm{Wallach}, \binits{N.R.}}:
\batitle{Globally hypoelliptic vector fields}.
\bjtitle{Topology}
\bvolume{12}(\bissue{3}),
\bfpage{247}--\blpage{253}
(\byear{1973})
\end{barticle}
\endbibitem

\bibitem[\protect\citeauthoryear{Forni}{}]{Forni08_cont-math}
\begin{botherref}
\oauthor{\bsnm{Forni}, \binits{G.}}:
On the greenfield-wallach and katok conjectures in dimension three. geometric
  and probabilistic structures in dynamics, 197213.
Contemp. Math
\textbf{469}
\end{botherref}
\endbibitem

\bibitem[\protect\citeauthoryear{Cardoso and Hounie}{1977}]{CH1977_pams}
\begin{barticle}
\bauthor{\bsnm{Cardoso}, \binits{F.}},
\bauthor{\bsnm{Hounie}, \binits{J.}}:
\batitle{Global solvability of an abstract complex}.
\bjtitle{Proc. Am. Math. Soc.}
\bvolume{65},
\bfpage{117}--\blpage{124}
(\byear{1977})
\doiurl{10.2307/2042004}
\end{barticle}
\endbibitem

\bibitem[\protect\citeauthoryear{Bergamasco and Petronilho}{1999}]{BP1999_jmaa}
\begin{barticle}
\bauthor{\bsnm{Bergamasco}, \binits{A.P.}},
\bauthor{\bsnm{Petronilho}, \binits{G.}}:
\batitle{Global solvability of a class of involutive systems}.
\bjtitle{J. Math. Anal. Appl.}
\bvolume{233}(\bissue{1}),
\bfpage{314}--\blpage{327}
(\byear{1999})
\doiurl{10.1006/jmaa.1999.6310}
\end{barticle}
\endbibitem

\bibitem[\protect\citeauthoryear{da~Silva and Meziani}{2016}]{DM2016_mana}
\begin{barticle}
\bauthor{\bsnm{Silva}, \binits{P.L.D.}},
\bauthor{\bsnm{Meziani}, \binits{A.}}:
\batitle{Cohomology relative to a system of closed forms on the torus}.
\bjtitle{Math. Nachr.}
\bvolume{289}(\bissue{17-18}),
\bfpage{2147}--\blpage{2158}
(\byear{2016})
\doiurl{10.1002/mana.201500293}
\end{barticle}
\endbibitem

\bibitem[\protect\citeauthoryear{de~{\'A}vila~Silva and
  Kirilov}{2019}]{AK2019_jst}
\begin{barticle}
\bauthor{\bsnm{{\'A}vila~Silva}, \binits{F.}},
\bauthor{\bsnm{Kirilov}, \binits{A.}}:
\batitle{Perturbations of globally hypoelliptic operators on closed manifolds}.
\bjtitle{J. Spectr. Theory}
\bvolume{9}(\bissue{3}),
\bfpage{825}--\blpage{855}
(\byear{2019})
\doiurl{10.4171/JST/264}
\end{barticle}
\endbibitem

\bibitem[\protect\citeauthoryear{Kirilov et~al.}{2021}]{KMP2021_jde}
\begin{barticle}
\bauthor{\bsnm{Kirilov}, \binits{A.}},
\bauthor{\bsnm{Paleari}, \binits{R.}},
\bauthor{\bsnm{Moraes}, \binits{W.A.A.}}:
\batitle{Global analytic hypoellipticity for a class of evolution operators on
  {{\(\mathbb{T}^1 \times \mathbb{S}^3\)}}}.
\bjtitle{J. Differ. Equations}
\bvolume{296},
\bfpage{699}--\blpage{723}
(\byear{2021})
\doiurl{10.1016/j.jde.2021.06.013}
\end{barticle}
\endbibitem

\bibitem[\protect\citeauthoryear{Vekua}{}]{vekua2014generalized}
\begin{botherref}
\oauthor{\bsnm{Vekua}, \binits{I.N.}}:
Generalized Analytic Functions
vol. 25,
Elsevier.
\doiurl{10.1016/C2013-0-05289-9}
\end{botherref}
\endbibitem

\bibitem[\protect\citeauthoryear{Bergamasco
  et~al.}{2014}]{bergamasco2014solvability}
\begin{barticle}
\bauthor{\bsnm{Bergamasco}, \binits{A.P.}},
\bauthor{\bsnm{Silva}, \binits{P.D.}},
\bauthor{\bsnm{Meziani}, \binits{A.}}:
\batitle{Solvability of a first order differential operator on the two-torus}.
\bjtitle{Journal of Mathematical Analysis and Applications}
\bvolume{416}(\bissue{1}),
\bfpage{166}--\blpage{180}
(\byear{2014})
\end{barticle}
\endbibitem

\bibitem[\protect\citeauthoryear{de~Almeida and Dattori~da
  Silva}{2021}]{de2021solvability}
\begin{barticle}
\bauthor{\bsnm{Almeida}, \binits{M.F.}},
\bauthor{\bsnm{Silva}, \binits{P.L.}}:
\batitle{Solvability of a class of first order differential operators on the
  torus}.
\bjtitle{Results in Mathematics}
\bvolume{76}(\bissue{2}),
\bfpage{104}
(\byear{2021})
\end{barticle}
\endbibitem

\bibitem[\protect\citeauthoryear{Ruzhansky and Turunen}{}]{RT2010_book}
\begin{botherref}
\oauthor{\bsnm{Ruzhansky}, \binits{M.}},
\oauthor{\bsnm{Turunen}, \binits{V.}}:
Pseudo-differential Operators and Symmetries: Background Analysis and Advanced
  Topics
vol. 2,
Springer Science \& Business Media.
\doiurl{10.1007/978-3-7643-8514-9}
\end{botherref}
\endbibitem

\bibitem[\protect\citeauthoryear{Kirilov et~al.}{2020}]{KMR2020_bsm}
\begin{barticle}
\bauthor{\bsnm{Kirilov}, \binits{A.}},
\bauthor{\bsnm{Moraes}, \binits{W.A.A.}},
\bauthor{\bsnm{Ruzhansky}, \binits{M.}}:
\batitle{Partial {Fourier} series on compact {Lie} groups}.
\bjtitle{Bull. Sci. Math.}
\bvolume{160},
\bfpage{27}
(\byear{2020})
\doiurl{10.1016/j.bulsci.2020.102853} .
\bcomment{Id/No 102853}
\end{barticle}
\endbibitem

\bibitem[\protect\citeauthoryear{Kirilov et~al.}{2021}]{KMR2021_jfa}
\begin{barticle}
\bauthor{\bsnm{Kirilov}, \binits{A.}},
\bauthor{\bsnm{Moraes}, \binits{W.A.A.}},
\bauthor{\bsnm{Ruzhansky}, \binits{M.}}:
\batitle{Global hypoellipticity and global solvability for vector fields on
  compact {Lie} groups}.
\bjtitle{J. Funct. Anal.}
\bvolume{280}(\bissue{2}),
\bfpage{39}
(\byear{2021})
\doiurl{10.1016/j.jfa.2020.108806} .
\bcomment{Id/No 108806}
\end{barticle}
\endbibitem

\bibitem[\protect\citeauthoryear{Kirilov et~al.}{2024}]{KKM2024}
\begin{barticle}
\bauthor{\bsnm{Kirilov}, \binits{A.}},
\bauthor{\bsnm{Kowacs}, \binits{A.P.}},
\bauthor{\bsnm{Moraes}, \binits{W.A.A.}}:
\batitle{Global solvability and hypoellipticity for evolution operators on tori
  and spheres}.
\bjtitle{Math. Nachr.}
\bvolume{297}(\bissue{12}),
\bfpage{4605}--\blpage{4650}
(\byear{2024})
\doiurl{10.1002/mana.202300506}
\end{barticle}
\endbibitem

\bibitem[\protect\citeauthoryear{de~Moraes}{2022}]{de2022regularity}
\begin{barticle}
\bauthor{\bsnm{Moraes}, \binits{W.A.A.}}:
\batitle{Regularity of solutions to a vekua-type equation on compact lie
  groups}.
\bjtitle{Annali di Matematica Pura ed Applicata (1923-)}
\bvolume{201}(\bissue{1}),
\bfpage{379}--\blpage{401}
(\byear{2022})
\end{barticle}
\endbibitem

\bibitem[\protect\citeauthoryear{Kirilov et~al.}{2026}]{kirilov2026solvability}
\begin{botherref}
\oauthor{\bsnm{Kirilov}, \binits{A.}},
\oauthor{\bsnm{Moraes}, \binits{W.A.A.}},
\oauthor{\bsnm{Tokoro}, \binits{P.M.}}:
Solvability of a class of evolution operators on compact lie groups.
arXiv preprint arXiv:2602.15203
(2026)
\end{botherref}
\endbibitem

\bibitem[\protect\citeauthoryear{da~Silva et~al.}{2025}]{da2025diagonal}
\begin{barticle}
\bauthor{\bsnm{Silva}, \binits{P.D.}},
\bauthor{\bsnm{Kirilov}, \binits{A.}},
\bauthor{\bsnm{Silva}, \binits{R.P.}}:
\batitle{Diagonal systems of differential operators on compact lie groups}.
\bjtitle{Results in Mathematics}
\bvolume{80}(\bissue{6}),
\bfpage{191}
(\byear{2025})
\doiurl{10.1007/s00025-025-02506-2}
\end{barticle}
\endbibitem

\end{thebibliography}

\end{document}